\documentclass[final,3p,times]{article}

\usepackage[a4paper,left=22mm,top=20mm,right=22mm,bottom=30mm]{geometry}

\usepackage{authblk}
\usepackage[numbers, sort&compress]{natbib}
\usepackage[margin=30pt,font=small,labelfont=bf]{caption}
\usepackage{amsmath,amsopn,amsthm,amssymb}

\usepackage{graphicx}
\usepackage{mathptmx} 
\usepackage[mathscr]{euscript}
\usepackage{natbib}
\usepackage{mathtools}
\usepackage{amssymb}
\usepackage{wasysym}
\usepackage{float}
\usepackage{enumitem}
\usepackage{xspace}
\usepackage{url}

\usepackage[
 	colorlinks=true,
	urlcolor=black,
	linkcolor=blue
]{hyperref}

\newenvironment{owndesc}%
    {\begin{description}[leftmargin = 0.2cm, labelsep = 0.2cm]}
    {\end{description}}

\makeatletter
\def\moverlay{\mathpalette\mov@rlay}
\def\mov@rlay#1#2{\leavevmode\vtop{%
    \baselineskip\z@skip \lineskiplimit-\maxdimen
    \ialign{\hfil$\m@th#1##$\hfil\cr#2\crcr}}}
\newcommand{\charfusion}[3][\mathord]{
  #1{\ifx#1\mathop\vphantom{#2}\fi
    \mathpalette\mov@rlay{#2\cr#3}
  }
  \ifx#1\mathop\expandafter\displaylimits\fi}
\DeclareRobustCommand\bigop[1]{%
  \mathop{\vphantom{\sum}\mathpalette\bigop@{#1}}\slimits@
}
\newcommand{\bigop@}[2]{%
  \vcenter{%
    \sbox\z@{$#1\sum$}%
    \hbox{\resizebox{\ifx#1\displaystyle.9\fi\dimexpr\ht\z@+\dp\z@}{!}{$\m@th#2$
}}%
  }%
}
\makeatother

\newcommand{\cupdot}{\charfusion[\mathbin]{\cup}{\cdot}}

\newtheorem{theorem}{Theorem}[section]
\newtheorem{lemma}[theorem]{Lemma} 
\newtheorem{corollary}[theorem]{Corollary}
\newtheorem{proposition}[theorem]{Proposition}

\newtheorem{observation}[theorem]{Observation}

\newcommand{\ovTheta}{\overline{\Theta}}
\newcommand{\diam}{\mathrm{diam}}
\newcommand{\join}{\mathrel{\ooalign{\hss$\triangleleft$\hss\cr$\triangleright$}}}

\providecommand{\keywords}[1]{\textbf{\textit{Keywords: }} #1}

\title{The Complement of the Djokovi\'{c}-Winkler Relation}

\author[1]{Marc Hellmuth} 

\author[2]{Bruno J.\ Schmidt}

\author[3]{Guillaume E.\ Scholz}

\author[1]{Sandhya Thekkumpadan Puthiyaveedu}

\affil[1]{Department of Mathematics, Faculty of Science,
  Stockholm University, SE-10691 Stockholm, Sweden }

\affil[2]{Max Planck Institute for Mathematics in the Sciences,
	Inselstra{\ss}e 22, D-04103 Leipzig, Germany}

\affil[3]{Bioinformatics Group, Department of Computer Science \&
    Interdisciplinary Center for Bioinformatics, Universit{\"a}t Leipzig,
    H{\"a}rtelstra{\ss}e~16--18, D-04107 Leipzig, Germany.}

\date{\ }

\begin{document}
\sloppy

\maketitle

\abstract{ 
The Djokovi\'{c}-Winkler relation $\Theta$ is a binary relation defined on the edge set of a given
graph that is based on the distances of certain vertices and which plays a prominent role in graph
theory. In this paper, we explore the relatively uncharted ``reflexive complement'' $\ovTheta$ of
$\Theta$, where $(e,f)\in \ovTheta$ if and only if $e=f$ or $(e,f)\notin \Theta$ for edges $e$ and
$f$. We establish the relationship between $\ovTheta$ and the set $\Delta_{ef}$, comprising the
distances between the vertices of $e$ and $f$ and shed some light on the intricacies of its
transitive closure $\ovTheta^*$. Notably, we demonstrate that $\ovTheta^*$ exhibits multiple
equivalence classes only within a restricted subclass of complete multipartite graphs. In addition,
we characterize non-trivial relations $R$ that coincide with $\ovTheta$ as those where the graph
representation is disconnected, with each connected component being the (join of) Cartesian product
of complete graphs. The latter results imply, somewhat surprisingly, that knowledge about the
distances between vertices is not required to determine $\ovTheta^*$. Moreover, $\ovTheta^*$ has
either exactly one or three equivalence classes.
}

\smallskip
\noindent

\section{Introduction}

The Djokovi\'{c}-Winkler relation $\Theta$ is a binary relation defined on the edge set of a given
graph. Two edges $e=\{a,b\}$ and $f=\{x,y\}$ of a graph $G$ are in relation $\Theta$ precisely if
\begin{equation}
d(a,x) + d(b,y) \neq d(a,y) + d(b,x), \label{eq:Theta}
\end{equation} 
where $d$ denotes the usual distance of two vertices in $G$. This relation was introduced by
Djokovi\'{c} \cite{DJOKOVIC:73} and specified in more detail later by Winkler
\cite{W84,WINKLER1987209}. The relation $\Theta$ is one of the most prominent relations when it
comes to studying Cartesian graph products \cite{HIK:11,FEIGENBAUM1985123,WINKLER1987209,Feder:92},
isometric embeddings into Cartesian products
\cite{HIK:11,W84,DJOKOVIC:73,graham1985isometric,winkler1987metric}, median graphs \cite{KS:00,
brevsar2010periphery, BRESAR2007345}, and many more, see e.g.\ \cite{NaSa:14, ImSa:98, KK:09}.

We study here the reflexive complement $\ovTheta$ of $\Theta$ which has -- to the best of our
knowledge -- not received any attention so far. We have $(e,f)\in \ovTheta$ precisely if $e=f$ or
$(e,f)\notin \Theta$ in which case equality holds in Equ.\ \eqref{eq:Theta}. As it turns out, the
property of subclasses of edges of $G$ being or not being in relation $\ovTheta$ can be used to
characterize certain graph classes, including block graphs, trees or graphs of diameter less than
three or greater than two. In particular, there is a close connection of the relation
$\ovTheta$ and the size of the set $\Delta_{e,f} = \{d(a,x), d(b,y), d(a,y) , d(b,x)\}$ comprising
the distances used to determine if $(e,f)\in \Theta$ or not. While $\ovTheta$ is reflexive and
symmetric, it is not transitive. Hence, it is of interest to understand its transitive closure
$\ovTheta^*$ in more detail. As it turns out, the condition that $e$ and $f$ are not contained in a
so-called $K_3$ or diamond is sufficient for edges $e$ and $f$ being in relation $\ovTheta^*$. We
furthermore characterize those graphs for which $\ovTheta^*$ has one, respectively more than one
equivalence class. As we shall see, $\ovTheta^*$ has more than one equivalence class only in a
surprisingly limited subclass of complete multipartite graphs. We are furthermore interested in
understanding the conditions under which an arbitrary relation $R$ coincides with $\ovTheta$ for
some graph. Under the assumption that $\ovTheta^*$ has more than one equivalence class, $R=\ovTheta$
for some graph if and only if the graph representation of $R$ has three connected components, and
each connected component is the Cartesian product of complete graphs or the join of the Cartesian
product of complete graphs. These results immediately imply that explicit knowledge about the
distance function $d$ in $G$ is not required in order to determine $\ovTheta^*$
and that $\ovTheta^*$ has either exactly one or three equivalence classes.

\section{Preliminaries}
\label{sec:prelim}

\subsection{Graphs}
We consider undirected simple graphs $G=(V,E)$, i.e., a tuples $(V,E)$ with non-empty vertex set
$V(G)\coloneqq V$ and set of edges $E(G)\coloneqq E$ consisting of two-elementary subsets of $V$.
Two edges $e=\{a,b\}$ and $f=\{x,y\}$ of $G$ are \emph{adjacent} if $\{a,b\} \cap \{x,y\} \neq
\emptyset$. Moreover, a vertex $v$ is \emph{incident} to an edge $e$ if $v\in e$. Two graphs $G$
and $H$ are \emph{isomorphic}, in symbols $G\simeq H$, if there is a bijection $\varphi \colon V(G)
\to V(H)$ such that $\{u,v\}\in E(G)$ if and only if $\{\varphi(u),\varphi(v)\}\in E(H)$ for all
$u,v\in V(G)$.

We often write, $x_1-x_2-\cdots - x_{n}$ for a path $P_n$ with vertices $x_1,\dots,x_n$ and edges
$\{x_i,x_{i+1}\}$, $1\leq i<n$ and call it \emph{$x_1x_n$-path}. A \emph{cycle} $C_n$ is a graph
with $n$ vertices for which the removal of any vertex results in a path. A graph $G$ is
\emph{connected} if there is a $uv$-path connecting $u$ and $v$ for all $u,v\in V(G)$. We denote
with $d_G(u,v)$ the distance between $u$ and $v$, i.e., the number of edges along shortest
$uv$-paths in $G$. We put $d_G(u,v)\coloneqq \infty$ if no such $uv$-path exists.

\begin{lemma}[{ \cite[L 2.1]{HTP:23}}]\label{lem:bip-dist}
	For every edge $\{a,b\}$ in a connected graph $G$ and every vertex $x\in
	V(G)$ it holds that $0\leq |d(x,a)-d(x,b)|\leq 1$. 
\end{lemma}

The value $\mathrm{diam}_G\coloneqq \max_{u,v\in V} d_G(u,v)$ denotes the diameter of $G$. A
subgraph $H$ of $G$ is \emph{isometric} if for all $x,y \in V(H)$, $d_H(x,y)=d_G(x,y)$. Clearly, if
$H$ is an isometric subgraph of $G$, then it is an \emph{induced} subgraph of $G$, i.e., $u,v\in
V(H)$ and $\{u,v\}\in E(G)$ implies $\{u,v\}\in E(H)$. The converse, however, does not necessarily
hold.

\begin{observation}\label{obs:diam2iso}
If $H$ is an induced subgraph of $G$ such that $\diam_H\leq 2$, then $H$ is an isometric subgraph of
$G$. 	
\end{observation}

We denote with $K_n=(V,E)$ the \emph{complete graph} on $n=|V|$ vertices, i.e., the graph that
satisfies $\{x,y\}\in E$ for all distinct $x,y\in V$. A \emph{complete multipartite graph}
$K_{n_1,\dots,n_\ell}=(V,E)$ is a graph whose vertex set $V$ can be partitioned into $\ell\geq 1$
sets $V_1,\dots,V_\ell$ with $1\leq n_i=|V_i|$, $1\le i\le \ell$ such that $\{x,y\}\in E$ if and
only if $x\in V_i$, $y\in V_j$ and $i\ne j$. Hence, each $V_i$ forms an independent set, i.e., the
subgraph induced by $V_i$ is edgeless. A \emph{star-tree} $S_{n} = (V,E)$ is a graph with $n+1$
vertices such that $E = \{\{v,x\}\mid v\in V\setminus \{x\}\}$ for some fixed $x\in V$. In other
words, for all distinct edges $e,f\in E$ it holds that $e\cap f=\{x\}$. Note that $S_1\simeq K_1$,
$S_2\simeq K_2$ and $S_3\simeq P_3$.

\begin{lemma}\label{lem:startree}
	Let $G=(V,E)$ be a connected graph. Then, $e\cap f\neq \emptyset$ for all $e,f\in E$ if and only
	if $G\simeq K_3$ or $G\simeq S_{|V|-1}$.  
\end{lemma}
\begin{proof}
	The \emph{if} direction is easy to verify. Suppose now that $e\cap f\neq \emptyset$ for all
	$e,f\in E$ of $G=(V,E)$. If $|V|\leq 3$, then connectedness of $G$ implies that $G\simeq S_1,S_2,
	S_3$ or $G\simeq K_{3}$ and we are done. Hence, assume that $|V|\geq 4$. Since $e\cap f\neq
	\emptyset$ for all $e,f\in E$ and $|V|\geq 4$ it follows that $G\not \simeq K_{|V|}$. Since $G$
	is connected and $G\not \simeq K_{|V|}$, there are at least two adjacent edges $e=\{v,x\},
	f=\{u,x\}\in E$ such that $\{u,v\}\not \in E$. Now let $f'=\{a,b\}\in E\setminus \{e,f\}$. Assume, 
	for contradiction, that $x\notin f'$. Since $e,f\neq f'$, we have
	$e\cap f' = \{v\}$ and $e\cap f' = \{u\}$. Hence, $f'=\{u,v\}\in E$; a contradiction.
	Therefore, $x\in f'$. As the latter holds for all edges $f'\in E\setminus \{e,f\}$ it follows
	that any two distinct edges intersect in $x$ only, which implies that $G\simeq S_{|V|-1}$. 
\end{proof}

\begin{figure}[ht]
\centering
\includegraphics[scale= 0.8]{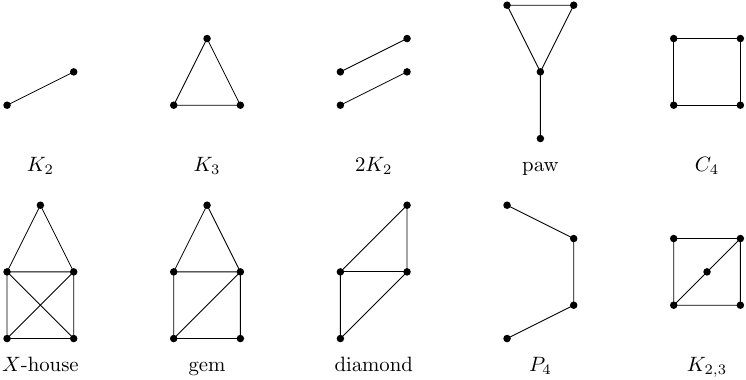}
\caption{Shown are a $K_2$, a $K_3$, a 2$K_2$, a paw, a $C_4$, an X-house, a gem, a diamond
         $K_{1,1,2}$, a $P_4$ and a $K_{2,3}$.
}	
\label{fig:graphs}
\end{figure}

We say that $G$ is $H$-free, if it does not contain $H$ as an induced subgraph.
In Fig.\ \ref{fig:graphs}, certain graphs $H$ are listed that play an important
role in this contribution. For later reference, we provide
\begin{theorem}[{\cite[Thm.\ 1]{OLARIU:88}}]\label{thm:paw-free}
	A graph $G$ is a paw-free if and only if each connected component of $G$ is $K_3$-free or
	complete multipartite. 
\end{theorem}

\emph{Block graphs}  \cite{H79, BM86}, also known as \emph{clique-trees}, 
are graphs in which each maximal biconnected component is a complete graph.
Block graphs are characterized in terms of isometric cycles as well as a 
so-called four-point condition. 
\begin{proposition}[\cite{BM86}]\label{prop:block}
The following statements are equivalent for every connected graph $G$.
\begin{itemize}[noitemsep,nolistsep]
\item[(i)] $G$ is a block graph.
\item[(ii)] For all $x,y,u,v \in V(G)$ distinct, the largest two of the sums $d(x,y)+d(u,v)$, $d(x,u)+d(y,v)$ and $d(x,v)+d(y,u)$ are equal.
\item[(iii)]  $G$ does neither contain a $C_n$ with $n\geq 4$ nor a diamond as an isometric subgraph.
\end{itemize}
\end{proposition}

The \emph{Cartesian product} $G\Box H$ of two graphs $G$ and $H$ has vertex set $V(G\Box H)
\coloneqq V(G)\times V(H)$ and edges $\{(a,x),(b,y)\}\in E(G\Box H)$ if and only if either $a=b$ and
$\{x,y\}\in E(H)$ or $\{a,b\}\in E(H)$ and $x=y$. The $K_1$ serves as unit element, that is,
$K_1\Box G\simeq G\Box K_1 \simeq G$ for all graphs $G$. The \emph{join} of two vertex-disjoint
graphs $G$ and $H$ is the graph $G\join H$ which has vertex set $V(G\join H) \coloneqq V(G) \cupdot
V(H)$ and edge set $E(G\join H) \coloneqq E(G) \cupdot E(H) \cupdot \{\{x,y\}\mid x\in V(G), y\in
V(H)\}$.

\subsection{The relation $\Theta$ and the set $\Delta$}

For a reflexive and symmetric binary relation $R$ we denote with $R^*$ its \emph{transitive
closure}, that is, the (unique) smallest transitive relation that contains $R$. Since we are
interested in the structure of certain symmetric relations $R\subseteq E\times E$, we define
$\mathscr G_R = (W,F)$ as the graph whose vertex set is $W=E$ and whose edge set $F$ comprises all
sets $\{x,y\}$ for which $(x,y)\in R$ and $x\neq y$. It is an easy task to verify $\mathscr G_R$
completely determines $R$ assuming that $R$ is symmetric and either reflexive or non-reflexive.
Furthermore, we say that $R^*$ is \emph{$1$-trivial}, if $R^*$ consists of one equivalence class
only and \emph{$|E|$-trivial} if $R^*$ consists of precisely $|E|$ equivalence classes. If $R$ is
defined on some graph $G$, we often write $R_G$ to make clear that $R$ is taken w.r.t.\ $G$.

To recall, the Djokovi\'{c}-Winkler relation $\Theta$ is a binary relation defined on the edge set
of a given graph $G$ such that $(e,f)\in \Theta$ for edges $e=\{a,b\}$ and $f=\{x,y\}$, precisely if
$d(a,x) + d(b,y) \neq d(a,y) + d(b,x)$. Note that $\Theta$ is reflexive and symmetric but not
necessarily transitive. Nevertheless, $\Theta^*$ is an equivalence relation.

\begin{lemma}[{\cite[L.\ 11.1 \& 11.3]{HIK:11}}]
No two distinct edges on a shortest path in a graph are in relation $\Theta$.
Moreover, suppose that a walk $P$ connects the endpoints of an edge $e$ but does not
contain $e$. Then $P$ contains an edge $f$ with $(e,f)\in\Theta$. 
\label{lem:SP}
\end{lemma}

\begin{observation}[{\cite[Sec.\ 11.1]{HIK:11}}]\label{obs:antipodal}
If $G = C_{2n}$ is a cycle of even length, then $\Theta$ consists of all pairs of antipodal edges. 
Hence, $\Theta^*$ has $n$ equivalence classes and $\Theta = \Theta^*$ in this case. 
On the other hand, any edge of an odd cycle $G=C_{2n+1}$ is in relation $\Theta$ with its two antipodal edges.
In this case, $\Theta \neq \Theta^*$ and $\Theta^*$ is $1$-trivial.
\end{observation}

For (not necessarily distinct) edges $e=\{a,b\}$ and $f=\{x,y\}$ of $G$ we put 
\[\Delta_{ef} \coloneqq \{d(a,x),  d(b,y),  d(a,y), d(b,x)\}.\]
Note that $\{d(a,x), d(b,y), d(a,y), d(b,x)\}$ is a set and thus, $|\Delta_{ef}|\in \{1,2,3,4\}$. As
shown next, $\Delta_{ef}$ is of a rather specific form and always satisfies $|\Delta_{ef}|\leq 3$.

\begin{lemma}\label{lem:furtherBasics}
	For every graph and all (not necessarily distinct) edges $e=\{a,b\}$ and $f=\{x,y\}$ of $G$ the
	following properties are satisfied.  
\begin{enumerate}	
	\item $1\leq |\Delta_{ef}|\leq 3$ and, in particular, $\Delta_{ef}\subseteq \{k,k+1,k+2\}$ for
	      some integer $k\geq 0$. Moreover, if $|\Delta_{ef}|=2$ then $\Delta_{ef}= \{k,k+1\}$.
	\item Suppose that $e\cap f \neq \emptyset$. If $e=f$ or $e\neq f$ and $e,f$ are contained in a
	      $K_3$, then $\Delta_{ef}=\{0,1\}$ and $(e,f)\in \Theta$. Otherwise, i.e., $e\neq f$ and
	      $e,f$ are not contained in a $K_3$, then $\Delta_{ef}=\{0,1,2\}$ and $(e,f)\notin \Theta$.
\end{enumerate}	
\end{lemma}
\begin{proof}
	Let $e=\{a,b\}$ and $f=\{x,y\}$ be (not necessarily distinct) edges of $G$. In what follows, we
	may assume w.l.o.g.\ that $d(a,x) = \min \{d(a,x), d(b,y), d(a,y), d(b,x)\}$. To see that
	$|\Delta_{ef}| = |\{d(a,x), d(b,y), d(a,y), d(b,x)\}|\leq 3$ assume, for contradiction, that
	$|\{d(a,x), d(b,y), d(a,y), d(b,x)\}| = 4$. Using the edge $\{a,b\}$ and vertex $x$ in Lemma
	\ref{lem:bip-dist} together with $d(a,x) < d(b,x)$ implies $d(a,x) = d(b,x)-1$. Using the edge
	$\{x,y\}$ and vertex $a$ in Lemma \ref{lem:bip-dist} together with $d(a,x) < d(a,y)$ implies
	$d(a,x) = d(a,y)-1$. Hence, $d(b,x) = d(a,y)$ and thus, $|\Delta_{ef}|\leq 3$; a contradiction.
	Therefore, $1\leq |\Delta_{ef}|\leq 3$ must hold. Put $d(a,x)=k$. By minimality of $d(a,x)$ and
	Lemma \ref{lem:bip-dist}, we have $d(b,x), d(a,y)\in \{k,k+1\}$ and $d(b,y)\in \{d(b,x)-1,d(b,x),
	d(b,x)+1\}$. This and $d(b,y)\geq k$ together with $d(b,x)\in \{k,k+1\}$ implies that $d(b,y)\in
	\{k,k+1,k+2\}$. Hence, $\Delta_{ef}\subseteq \{k,k+1,k+2\}$ for some integer $k\geq 0$. Assume
	that $|\Delta_{ef}|=2$. If $d(b,x)=k$, then $d(b,y)\geq k$ implies that $d(b,y)\in \{k,k+1\}$ and
	thus, $\Delta_{ef}= \{k,k+1\}$. If $d(b,x)=k+1$, then $d(b,y)\neq k+2$ as otherwise,
	$|\Delta_{ef}|=3$. Hence, $d(b,y) \in \{k,k+1\}$ and we obtain $\Delta_{ef}= \{k,k+1\}$.

	Suppose now that $e\cap f\neq \emptyset$. Clearly, if $e=f$, then $\Delta_{ef}=\{0,1\}$ and
	$(e,f)\in \Theta$. Assume that $e\neq f$. W.l.o.g.\ assume that $b=x$. If $e,f$ are contained in
	a $K_3$, then $d(a,x) = d(b,y) = d(a,y) = 1$ and $d(b,x) = 0$. Thus, $\Delta_{ef}=\{0,1\}$ and
	$(e,f)\in \Theta$. If there is no edge $\{a,y\}$, then $e\cup f$ induces a shortest $ay$-path of
	length two. Hence, $\Delta_{ef}=\{0,1,2\}$ and, by Lemma \ref{lem:SP}, $(e,f)\notin \Theta$.
\end{proof}

For later reference, we give here a generalization of a result that was shown for
bipartite graphs in \cite[L.\ 11.2]{HIK:11}.
 
\begin{lemma}
Let $G=(V,E)$ be a graph and $e,f\in E$ be edges that satisfy $(e,f)\in \Theta$. Then, $\Delta_{ef}
= \{k,k+1\}$ for some $k\geq 0$ and thus, $|\Delta_{ef}|=2$. 
\label{lem:theta-delta}
\end{lemma} 
\begin{proof}
Let $e,f\in E$ be edges that satisfy $(e,f)\in \Theta$. If $e=f$, then $\Delta_{ef}=\{0,1\}$ and,
therefore, $|\Delta_{ef}|=2$. Hence, suppose that $e\neq f$. If $e$ and $f$ share a common vertex,
then $(e,f)\in \Theta$ together with Lemma \ref{lem:SP} implies that $e$ and $f$ must be contained
in a $K_3$. By Lemma~\ref{lem:furtherBasics}, $\Delta_{ef}=\{0,1\}$ and thus, $|\Delta_{ef}|=2$. 
  
Suppose now that $e=\{x,y\}$ and $f=\{a,b\}$ are non-adjacent. Without loss of generality, suppose
that $d(a,x)=\min\{d(a,x), d(b,y), d(a,y), d(b,x)\}$. Put $k=d(a,x)$. Lemma~\ref{lem:bip-dist},
together with the minimality of $d(a,x)$, implies that $d(a,y) \in \{k,k+1\}$ and $d(b,x) \in
\{k,k+1\}$. Using Lemma~\ref{lem:bip-dist} again together with the fact that $d(b,y)\geq k$, 
we obtain $d(b,y) \in \{k,k+1,k+2\}$. Suppose
first that $d(b,y)=k+2$. Then by Lemma~\ref{lem:bip-dist}, we obtain $d(a,y)=d(b,x)=k+1$. It
follows, that $d(a,x)+d(b,y)=2k+2=d(a,y)+d(b,x)$, a contradiction to $(e,f) \in \Theta$. Hence,
$d(b,y) \in \{k,k+1\}$ must hold. In particular, we have $|\Delta_{ef}| \leq 2$. Since $(e,f) \in
\Theta$, $|\Delta_{ef}|=1$ cannot hold, so $|\Delta_{ef}|=2$ follows. In particular, $\Delta_{ef} =
\{k,k+1\}$.
\end{proof}

The next results summarizes useful properties provided by the size of the set $\Delta$.

\begin{proposition}\label{pr:Delta-distinctEdges}
The following statements hold for every connected graph $G=(V,E)$ where $E\neq \emptyset$. 
\begin{enumerate}
	\item $|\Delta_{ef}|  \neq 3$ for all (distinct)  $e,f\in E$  $\iff$  $G$ is a complete graph.  

	\item $|\Delta_{ef} | =1$ for all distinct $e,f\in E$ $\iff $ $G$ is a $K_2$. 
	
			In other words, there is no connected graph that contains at least two edges and for which
			$|\Delta_{ef} | =1$ hold for all distinct $e,f\in E$.
	
	\item $|\Delta_{ef} | =2$ for all (distinct) $e,f\in E$ $\iff$ $G$ is a $K_2$ or $K_3$ $\iff$
	      $\Theta_G = \Theta_G^*$ and $\Theta_G^*$ is $1$-trivial. 
		 
	\item $\begin{aligned}[t]
					|\Delta_{ef} |=3 \text{ for all distinct } e,f\in E &\iff |\Delta_{ef} |\neq 2 \text{ for all distinct } e,f\in E
									\\ &\iff \Theta_G = \Theta_G^* \text{ and } \Theta_G^* \text{ is } |E|\text{-trivial} \iff G \text{ is a tree.}	  
			  \end{aligned}$		
\end{enumerate}
\end{proposition}
\begin{proof}
	Observe first that for every edge $e$ it holds that $\Delta_{ee} = \{0,1\}$ and, thus,
	$|\Delta_{ee}|=2\neq 3$. For Condition (1) suppose first that $G=(V,E)$ is a complete graph and
	let $e,f\in E$. If $e=f$, then $|\Delta_{ef} | \neq 3$. Hence, assume that $e=\{a,b\}\neq f=\{x,y\}$.
	If $e\cap f\neq \emptyset$ then $C\simeq K_n$ implies that the edges $e,f$ are contained in a
	$K_3$ and, by Lemma \ref{lem:furtherBasics}, $|\Delta_{ef}|=2\neq 3$. Suppose now that $e\cap f=
	\emptyset$. Since $G$ is a complete graph, it holds that $\Delta_{ef} = \{1\}$ and thus,
	$|\Delta_{ef} | \neq 3$. For the converse assume, by contrapositioning, that $G$ is not a
	complete graph and thus, distinct from a $K_2$. Since $G$ is connected and $E\neq \emptyset$,
	there are two edges $e\neq f$ that share a common vertex but are not contained in a common $K_3$.
	Lemma \ref{lem:furtherBasics} implies $|\Delta_{ef}|=3$. Thus Condition (1) holds. 

	For Condition (2), if $G$ is a $K_2$, then the statement is vacuously true. Conversely, if
	$G$ is not a $K_2$, it must contain two incident edges since $G$ is connected and $E\neq
	\emptyset$. By Lemma \ref{lem:furtherBasics}, $|\Delta_{ef} | \neq 1$. Thus Condition (2) holds. 

	For Condition (3) suppose first that for all (distinct) edges $e$ and $f$ of $G$ it holds that
	$|\Delta_{ef} | =2$. By contradiction, assume that $G$ is not a $K_2$ or $K_3$. This and
	Condition (1) implies that $G\simeq K_n$, $n>3$. Hence, there are two vertex-disjoint edges
	$e\neq f$ such that $\Delta_{ef} = \{1\}$; a contradiction to $|\Delta_{ef} | =2$. Hence, $G$ is
	a $K_2$ or $K_3$. Moreover, if $G$ is a $K_2$ or $K_3$ the one easily verifies that $\Theta_G =
	\Theta_G^*$ and $\Theta_G^*$ is $1$-trivial. Furthermore, if $\Theta_G =
	\Theta_G^*$ and $\Theta_G^*$ is $1$-trivial, then $(e,f)\in \Theta$ for all edges $e$ and $f$ of
	$G$. This together with Lemma \ref{lem:theta-delta} implies that $|\Delta_{ef} | =2$ for all for
	all edges $e$ and $f$ of $G$. In summary, Statement (3) is always satisfied. 
	
	For Condition (4), clearly $|\Delta_{ef} |=3$ implies $|\Delta_{ef} |\neq 2$ for all distinct
	$e,f\in E$. Suppose that $|\Delta_{ef} |\neq 2$ for all distinct $e,f\in E$. Contraposition of
	Lemma \ref{lem:theta-delta} implies $(e,f)\notin \Theta$ for all distinct $e,f\in E$. Hence,
	$\Theta_G = \{(e,e)\mid e\in E\}$ and thus, $\Theta_G = \Theta_G^*$ and $\Theta_G^*$ is
	$|E|$-trivial. Assume now that $\Theta_G = \Theta_G^*$ and $\Theta_G^*$ is $|E|$-trivial. Assume,
	for contradiction, that $G$ contains a cycle $C$. In this case, Lemma \ref{lem:SP} implies that
	there are two edges $e,f$ in $C$ such that $(e,f)\in \Theta_G$; a contradiction. Hence, $G$ does
	not contain cycles. As $G$ is connected, it must be a tree. Assume now that $G$ is a tree. Let
	$e=\{a,b\}$ and $f = \{x,y\}$ be two distinct edges of $G$. We can assume w.l.o.g.\ that
	$d_G(a,y) = \min \Delta_{ef}$. Since $G$ is a tree, there is a unique shortest $xb$-path $P$.
	Since $k\coloneqq d_G(a,y) = \min \Delta_{ef}$, the edges $e$ and $f$ are contained in $P$. One
	easily verifies that $d_G(x,a) = d_G(y,b) = k+1$ and $d_G(x,b)= k+2$ must hold and, therefore,
	$|\Delta_{ef}|=3$. 
\end{proof}

Proposition \ref{pr:Delta-distinctEdges}(3) implies:

\begin{corollary}	
Suppose that $\Theta_G = \Theta_G^*$. Then, $\Theta_G^*$ is not $1$-trivial if and only if there are
two edges $e$ and $f$ in $G$ with $|\Delta_{ef}| \neq 2$.
\end{corollary}

\section{The relation $\overline{\boldsymbol{\Theta}}$}

We are in particular interested in the reflexive complement of $\Theta_G$ and put, for a given graph
$G=(V,E)$, \[\ovTheta_G \coloneqq \{(e,f) \mid e=f \text{ or } e\neq f \text{ and } (e,f)\notin
\Theta \text{ for all } e,f\in E\}.\] By definition, $\Theta_G \cap \ovTheta_G = \{(e,e)\mid e\in
E\}$. Note that $\ovTheta_G$ is reflexive and symmetric and we have $(e,f)\in \ovTheta_G$ for
distinct $e=\{a,b\}$ and $f=\{x,y\}$ precisely if \[d(a,x) + d(b,y) = d(a,y) + d(b,x). \]
 
We start with a few simple observations.  
\begin{lemma}\label{lem:path}
Let $G$ be a graph and let $P$ be an induced path of $G$. Then, $(e,f) \in \ovTheta_G^*$
for all edges $e,f$ in  $P$.
\end{lemma}
\begin{proof}
Let $P$ be an induced path of $G$. Observe first that any two adjacent $e' = \{x,u\},f'=\{y,u\}$ in
$P$ form a shortest $xy$-path since $P$ is induced and thus, no edge $\{x,y\}\in E(G)$ exists. This
together with Lemma \ref{lem:SP} and the fact that $\ovTheta_G$ is reflexive
implies that all adjacent edges in $P$ are in relation $\ovTheta_G$. It
is now a straightforward task to verify that all edges $e$ and $f$ of $P$ satisfy $(e,f)
\in \ovTheta^*$.
\end{proof}

\begin{lemma}\label{lem:cut-edge}
If $G=(V,E)$ contains a cut-edge $e$, i.e., $(V,E\setminus \{e\})$ has more connected components
than $G$, then $\ovTheta^*_G$ is $1$-trivial.
\end{lemma}
\begin{proof}
Let $e=\{a,b\}$ be a cut-edge in $G$. For any edge $f=\{u,v\}$ in $G$ that is not located in the
connected component of $G$ that contains $e$ it holds that $d(a,u) + d(b,v) = \infty = d(a,v) +
d(b,u)$ and, therefore, $(e,f)\in \ovTheta\subseteq \ovTheta^*$. Suppose that there is an edge
$f=\{u,v\}$ that is contained in the connected component $C$ of $G$ that contains $e$. Put
$H=(V,E\setminus \{e\})$. Since $e$ is a cut-edge and $e$ is located in $C$, the induced subgraph
$H[C]$ of $H$ must be disconnected. In particular, $H[C]$ decomposes into two connected components
$C_a$ and $C_b$ that contain $a$ and $b$, respectively. Assume, there is an edge $f=\{u,v\}\in C_b$.
Since $a$ is not adjacent to any vertex in $C_b$, it follows that any shortest path in $G$ from $a$
to any vertex in $C_b$ must contain vertex $b$. Thus, $d_G(a,u)=d_G(b,u)+1$ and
$d_G(a,v)=d_G(b,v)+1$ and, therefore, $d_G(a,u)+d_G(b,v)=(d_G(b,u)+1)+(d_G(a,v)-1)=d(b,u)+d(a,v)$.
Hence, $(e,f)\in \ovTheta\subseteq \ovTheta^*$. By similar arguments, if $f$ is located in $C_a$,
then $(e,f)\in \ovTheta^*$. In summary, all edges $f$ in $G$ satisfy $(e,f)\in \ovTheta^*$ which
immediately implies that $\ovTheta^*_G$ is $1$-trivial.
\end{proof}

\begin{lemma}\label{lem:Theta*1t}
Let  $G=(V,E)$ be a  graph and  $\{R,P\} =\{\Theta_G, \ovTheta_G\}$. 
	Then, the following statements are satisfied. 
	\begin{enumerate}
		\item If $R^*$ is not $1$-trivial, then $P^*$ is $1$-trivial. 
		\item if $R = R^*$ and $R^*$   is neither $1$- nor $|E|$-trivial, then $P^*$ is $1$-trivial	and $P\neq P^*$.
	\end{enumerate}
\end{lemma}
\begin{proof}
Let $\{R,P\} =\{\Theta_G, \ovTheta_G\}$. Note that the number of equivalence classes of $R^*$ and
$P^*$ is precisely the number of connected components of $\mathscr{G}_{R}$ and $\mathscr{G}_{P}$,
respectively. Since $R^*$ is not $1$-trivial, $\mathscr{G}_{R}$ has $k>1$ connected components
$C_1,\dots,C_k$. Moreover, the complement $\overline{\mathscr{G}}_{R}$ of $\mathscr{G}_{R}$
satisfies $\overline{\mathscr{G}}_{R} = \mathscr{G}_{P}$. Since for $i,j \in \{1, \ldots, k\}$
distinct, every vertex in $C_i$ is adjacent to every vertex in $C_j$ in $\overline{\mathscr{G}}_{R}$
it follows that $\mathscr{G}_{P}$ is connected. Thus, $P^*$ has one equivalence class and is,
therefore, $1$-trivial.

Assume now that $R = R^*$ and $R^*$ is neither $1$- nor $|E|$-trivial. By the arguments, above,
$P^*$ is $1$-trivial. Let $E_1,\dots,E_k$ be the equivalence classes of $R = R^* $ where $1<k<|E|$.
By definition, $(e,f)\in R$ if and only if $(e,f)\notin P$ for all distinct edges $e,f\in E$. In
particular, since $R$ is an equivalence relation, we have $(e,f)\in R$ for all $e,f\in E_i$ and,
furthermore, $(e,f)\notin R$ for all $e\in E_i$, $f\in E_j$, $i\neq j$. Hence, $P = \{(e,f)\mid e=f
\text{ or } e\in E_i, f\in E_j, i\neq j \}$. Note that $k<|E|$ implies that there is an $i \in \{1,
\ldots, k\}$ such that $E_i$ contains two distinct edges $e$ and $f$, and $1<k$ implies that there
is an equivalence class $E_j$ distinct from $E_i$ and thus, some $g\in E_j$ with $g\neq e,f$. By
definition, $(e,g),(f,g) \in P$ and thus, $(e,f) \in P^*$, but $(e,f)\notin P$. Therefore, $P\neq
P^*$. 
\end{proof}

The converse of Lemma~\ref{lem:Theta*1t}(1) is in general not satisfied, i.e., there
are connected graphs for which $\Theta_G^*$ and $\ovTheta_G^*$ are both $1$-trivial. To see this consider
$G=C_n$ for some odd $n\geq 5$. By Lemma \ref{obs:antipodal}, $\Theta^*$ is $1$-trivial. Moreover,
any two adjacent edges in $G$ induce always a path and are, therefore, in relation $\ovTheta^*$ by
Lemma~\ref{lem:path}. One easily observes that this implies that $\ovTheta^*$ is $1$-trivial. 	

\begin{lemma}\label{lem:trivial}
The following statements are equivalent for every graph $G$.
\begin{enumerate}
	\item $\Theta_G = \Theta^*_G$ and $\Theta_G^*$ is $1$-trivial  (resp., $|E|$-trivial)
	\item $\ovTheta_G = \ovTheta^*_G$ and $\ovTheta^*_G$ is $|E|$-trivial  (resp., $1$-trivial). 
	\item $G$ is a $K_2$ or $K_3$ (resp., a tree)
	\item For all distinct edges $e,f$ of $G$ it holds that $|\Delta_{ef}|=2$ (resp.,
	      $|\Delta_{ef}|\neq 2$ or equivalently $|\Delta_{ef}| = 3$) 
\end{enumerate}
\end{lemma} 
\begin{proof}
	Observe first that (1), (3) and (4) are equivalent by Proposition~\ref{pr:Delta-distinctEdges}.
	 Suppose that (3) holds. If $G$ is a $K_2$ or $K_3$ one easily verifies that $\ovTheta_G =
	 \ovTheta^*_G$ and $\ovTheta^*_G$ is $|E|$-trivial. Furthermore, if $G$ is a tree, then any two
	 distinct edges $e,f$ are located on some shortest path $P$ that contains $e$ and $f$. By Lemma
	 \ref{lem:SP}, $(e,f)\in \ovTheta_G$. In addition, $P$ must be induced and thus, by
	 Lemma~\ref{lem:path}, $(e,f)\in \ovTheta_G^*$, i.e., $\ovTheta_G = \ovTheta_G^*$ and
	 $\ovTheta^*_G$ is $|1|$-trivial. Hence, (3) implies (2). 
	
	Suppose that (2) holds. Assume first that $\ovTheta_G = \ovTheta^*_G$ and $\ovTheta^*_G$ is
	$1$-trivial. Since $\ovTheta_G$ is an equivalence relation with one equivalence class only, it
	follows that $(e,f)\in \ovTheta_G$ for all distinct edges $e$ and $f$ of $G$. Hence, $(e,f)\notin
	\Theta_G$ for all distinct edges $e$ and $f$ of $G$. Since $\Theta_G$ is reflexive, it follows that
	$\Theta_G = \{(e,e)\mid e\in E\}$. Hence, $\Theta_G = \Theta^*_G$ and $\Theta^*_G$ is
	$|E|$-trivial. Suppose that $\ovTheta_G = \ovTheta^*_G$ and $\ovTheta^*_G$ is $|E|$-trivial. It
	follows that $\ovTheta_G = \{(e,e)\mid e\in E\}$. Hence, $(e,f)\notin \ovTheta_G$ and thus,
	$(e,f)\in \Theta_G$ for all distinct edges $e$ and $f$ of $G$. Together with reflexivity of
	$\Theta_G$ this implies that $\Theta_G = \Theta^*_G$ and $\Theta^*_G$ is $1$-trivial. Hence (2)
	implies (1) which completes this proof.
\end{proof}

In case a graph contains non-adjacent edges, we can relax Condition (4) in Lemma \ref{lem:trivial} 
by considering non-adjacent edges only. 

\begin{lemma}\label{lem:trivial2} 
	The following two statements are equivalent for every connected graph $G$ that contains 
	at least two non-adjacent edges. 
	\begin{enumerate}
		\item $\ovTheta_G = \ovTheta_G^*$ and $\ovTheta_G$ is $1$-trivial. 
		\item $|\Delta_{ef}|= 3$ for all non-adjacent edges $e$ and $f$ of $G$. 
		\item $G$ is a tree. 
	\end{enumerate}
\end{lemma}
\begin{proof}
	If $\ovTheta_G = \ovTheta_G^*$ and $\ovTheta_G$ is $1$-trivial, then 
	By	Lemma \ref{lem:trivial}, $|\Delta_{ef}| = 3$ for all distinct and thus, in particular,
	for all non-adjacent edges $e$ and $f$. Thus, Condition (1) implies (2).  
	
	Assume now that $|\Delta_{ef}|= 3$ for all non-adjacent edges $e$ and $f$ of $G$. We continue
	with showing that $G$ is a tree. Since $G$ contains non-adjacent edges, $G\not \simeq K_3$ must
	hold. Assume, for contradiction, that $G$ contains cycles and, thus, in particular, induced
	cycles. Let $C$ be an induced cycle. Suppose first that that $|V(C)|\geq 4$ and let $e \in C$. By
	Lemma \ref{lem:SP}, there exists an edge $f$ in $C\setminus \{e\}$ such that $(e,f)\in \Theta$.
	Lemma \ref{lem:theta-delta} implies that $|\Delta_{ef}|=2$. Note that $e$ and $f$ must be
	adjacent since, otherwise, $|\Delta_{ef}|=3$ would hold by assumption. Hence, $e=\{x,u\}$ and
	$f=\{u,y\}$ for some pairwise distinct vertices $u,x,y\in V(G)$. Since $e$ and $f$ share a common
	vertex and $(e,f)\in \Theta$, Lemma \ref{lem:SP} implies that $e$ and $f$ must be contained in a
	$K_3$. Hence, $\{x,y\}\in E(G)$. Since $x$ and $y$ are located on $C$ and since $C$ contains more
	than three vertices, it follows that $C$ is not induced; a contradiction. Hence, $|V(C)| = 3$
	must hold, i.e., $C\simeq K_3$. Since $G\not\simeq K_3$ and since $G$ is connected, there is a
	vertex $v$ in $V(G)\setminus V(C)$ that is adjacent to at least one vertex in $V(C)$. If $v$ is
	adjacent to exactly one (resp., exactly two, three) vertices in $V(C)$, then $V(C)\cup \{v\}$
	induces a paw (resp., a diamond, $K_4$). In all of these cases, one easily observes that the
	subgraph induced by $V(C)\cup \{v\}$ contains two non-adjacent edge $e,f$ with $\Delta_{ef}=
	\{1,2\}$ or $\Delta_{ef} =\{1\}$ and thus, $|\Delta_{ef}|\neq 3$; a contradiction. Thus, $G$ does
	not contain any induced cycles and thus, no cycles at all. Since $G$ is connected it is,
	therefore, a tree and Condition (2) implies (3).
	
	Finally, if $G$ is a tree, we can apply
	Lemma \ref{lem:trivial} to conclude that $\ovTheta_G = \ovTheta_G^*$ and $\ovTheta_G$ is $1$-trivial. 
	Thus, Condition (3) implies (1) which completes this proof. 
\end{proof}

In Lemma \ref{lem:trivial}, $\ovTheta_G = \ovTheta^*_G$ and $\ovTheta^*_G$ being $1$-trivial is
equivalent to $(e,f)\in \ovTheta_G$ for all distinct edges $e$ and $f$ in $G$. The latter is equivalent to
$G$ being a tree. However, if we relax this condition by assuming that only non-adjacent edges $e$
and $f$ in $G$ satisfy $(e,f)\in \ovTheta_G$, then we obtain a simple and novel characterization of
block graphs.

\begin{proposition}\label{prop:blockgraph-Delta}
The following statements are equivalent for every connected graph $G=(V,E)$. 
\begin{enumerate}
	\item $(e,f)\in \ovTheta_G$ for all non-adjacent edges $e,f$ of $G$. 
	\item $G$ is a block graph.
\end{enumerate}
\end{proposition}
\begin{proof}
Suppose that $(e,f)\in \ovTheta_G$ for all non-adjacent edges $e,f$ of $G$. Clearly, $G$ cannot
contain an isometric diamond $H$ since then $H$ (and, therefore $G$) has two non-adjacent edges
$e,f$ with $(e,f)\in \Theta_G$ and, thus $(e,f)\notin \ovTheta_G$. Similarly, $G$ cannot contain an
isometric cycle since then antipodal edges $e,f$ of this cycle satisfy $(e,f)\in \Theta_G$ (cf.\
Obs.\ \ref{obs:antipodal}). By Prop.~\ref{prop:block}, $G$ is a block graph. 

Conversely, suppose that $G$ is a block graph, and let $e=\{x,y\}$ and $f=\{u,v\}$ be two
non-adjacent edges of $G$. By Prop.~\ref{prop:block}, the largest two of the sums $d(x,y)+d(u,v)$,
$d(x,u)+d(y,v)$ and $d(x,v)+d(y,u)$ are equal. Since $\{x,y\}$ and $\{u,v\}$ are edges of $G$, we
have $d(x,y)+d(u,v)=2$. Moreover, the vertices $x,y,u$ and $v$ are pairwise distinct, so both
$d(x,u)+d(y,v) \geq 2$ and $d(x,v)+d(y,u) \geq 2$ must hold. In particular, we have $d(x,y)+d(u,v)
\leq d(x,u)+d(y,v)$ and $d(x,y)+d(u,v) \leq d(x,v)+d(y,u)$. Hence, $d(x,u)+d(y,v)=d(x,v)+d(y,u)$
must hold. Thus,  $(e,f)\in \ovTheta_G$. 
\end{proof}

The results in Proposition~\ref{pr:Delta-distinctEdges} for the set $\Delta$ are mainly constrained
by the fact that the restrictions must, in particular, hold for adjacent edges as well. As
Prop.~\ref{prop:blockgraph-Delta} shows, considering non-adjacent edges instead helps to
characterize block graphs. In particular, $(e,f)\in \ovTheta_G$ for all non-adjacent edges $e,f$ of
$G$ holds in this case. We thus continue with considering properties of graphs based on non-adjacent
edges and characterize graphs for which $(e,f)\notin \ovTheta_G$ for all non-adjacent edges $e,f$ of
$G$. We start with the following simple result which characterizes graphs with diameter at most $2$
(as well as those with diameter at least $3$) in terms of the size of $\Delta$.

\begin{proposition}\label{prop:DN3-nonAdj}
	The following statements are equivalent for every connected graph $G$.
		\begin{enumerate}
		\item  $|\Delta_{ef}|\neq 3$ for all non-adjacent edges $e$ and $f$ of $G$ 
		\item  $\diam_G\leq 2$
		\item  $\Delta_{ef}\subseteq\{1,2\}$ for all non-adjacent edges $e$ and $f$ of $G$.
	\end{enumerate}
	Thus,  $\diam_G\geq 3$ if and only if there are non-adjacent edges $e$ and $f$ in $G$ with $|\Delta_{ef}| = 3$.
\end{proposition}
\begin{proof}
Clearly (3) implies (1) and we continue with showing that (1) implies (2) and that (2) implies (3).
Suppose that $|\Delta_{ef}|\neq 3$ for all non-adjacent edges $e$ and $f$ of $G$. Assume, for
contradiction, that $\diam_G\geq 3$ and thus, there are vertices $x$ and $y$ in $G$ such that
$d_G(x,y)\geq 3$. Consider a shortest $xy$-path $P = x_1-x_2-\ldots-x_k$ in $G$ where $x_1=x$ and
$x_k=y$. Since $d_G(x,y)\geq 3$, we have $k\geq 4$ and, in particular, $e\cap f=\emptyset$ where
$e=\{x_1,x_2\}$ and $f= \{x_{k-1},x_k\}$, i.e., $e$ and $f$ are non-adjacent. Moreover, it holds
that $d_G(x_1,x_{k})=k-1$, $d_G(x_1,x_{k-1})=d_G(x_2,x_k) = k-2$ and $d_G(x_2,x_{k-1})=k-3$. Hence
$\Delta_{ef} =\{k-3,k-2,k-1\}$ for some $k\geq 4$ and, thus $|\Delta_{ef}| = 3$; a contradiction. In
summary, (1) implies (2). Suppose now that $\diam_G\leq 2$. Hence, $d_G(x,y)\in \{0,1,2\}$ for all
vertices $x,y\in V(G)$. Let $e$ and $f$ be two non-adjacent edges in $G$. Thus, we have $0 \notin
\Delta_{ef}$ which implies that $\Delta_{ef} \subseteq \{1,2\}$ for all non-adjacent edges $e$
and $f$. Hence, (2) implies (3). Negation of (1) and (2) yields the last statement.
\end{proof}

\begin{proposition}\label{prop:nonAdj-notTheta}
	The following statements are equivalent for every connected graph $G$.
	\begin{enumerate}
			\item $(e,f)\not\in \ovTheta_G$ and thus, $(e,f)\in \Theta_G$ for all non-adjacent edges $e$ and $f$ of $G$.
	      \item $|\Delta_{ef} | =2$ for all non-adjacent edges $e,f$ of $G$ and $G$ is paw-free.
			\item $\Delta_{ef} = \{1,2\}$ for all non-adjacent edges $e$ and $f$ of $G$ and $G$ is paw-free.
	      \item $\diam_G \leq 2$ and $G$ is $\{K_4,2K_2,\text{paw}\}$-free.
		 \end{enumerate}
\end{proposition}
\begin{proof}
Let $G$ be a connected graph. Assume that Condition (1) holds. Since $(e,f)\not\in \ovTheta_G$ for
all non-adjacent edges $e$ and $f$ of $G$, $(e,f)\in \Theta_G$ must hold for all non-adjacent edges
$e$ and $f$ of $G$. Lemma \ref{lem:theta-delta} implies that $|\Delta_{ef} | =2$ for all
non-adjacent edges $e,f$. Moreover, $G$ cannot contain an induced paw, since a paw contains
non-adjacent edges $e=\{x,y\}$ and $f=\{a,b\}$ satisfying $d_G(x,a) + d_G(y,b) = 3 = d_G(x,b) +
d_G(y,a)$ and thus, $(e,f)\in \ovTheta_G$. Hence, Statement (1) implies (2). 
	
Assume now that $|\Delta_{ef} | =2$ for all non-adjacent edges $e,f$ of $G$ and that $G$ is
paw-free. By Prop.~\ref{prop:DN3-nonAdj}(3), $\Delta_{ef} = \{1,2\}$ holds for all non-adjacent edges
$e$ and $f$ of $G$. Thus, Statement (2) implies (3). 

Assume now that $\Delta_{ef} = \{1,2\}$ for all non-adjacent edges $e$ and $f$ of $G$ and that $G$
is paw-free. Clearly, $G$ cannot contain an induced $K_4$, since then there are non-adjacent edges
$e,f$ satisfying $\Delta_{ef} = \{1\}$. Moreover, $G$ cannot contain an induced 2$K_2$ since then
there are non-adjacent edges $e,f$ satisfying $1\notin \Delta_{ef}$. Hence, $G$ is
$\{K_4,2K_2,\text{paw}\}$-free. Note that $\Delta_{ef} = \{1,2\}$ implies $|\Delta_{ef}|\neq 3$ for
all non-adjacent edges $e$ and $f$ of $G$. Hence, we can apply Prop.~\ref{prop:DN3-nonAdj} 
to conclude that $\diam_G \leq 2$. In summary, Statement (3) implies (4). 

Assume now that $G$ satisfies the conditions in Statement (4), i.e., $\diam_G \leq 2$ and $G$ is
$\{K_4,2K_2,\text{paw}\}$-free. Suppose that $G$ contains non-adjacent edges $e=\{x,y\}$ and
$f=\{u,v\}$. Let $\ell$ denote the number of edges $\{a,b\}$ with $a\in \{x,y\}$ and $b\in \{u,v\}$.
Clearly, $\ell\in \{0,1,2,3,4\}$. However, $\ell=0$ and $\ell=4$ is not possible, since $G$ does not
contain induced 2$K_2$s and $K_4$s. Since $\ell\geq 1$ there is at least one edge between the
vertices in the sets $\{x,y\}$ and $\{u,v\}$. W.l.o.g., assume that the edge $g=\{x,u\}$ exists.
Consider first the case that $\ell=1$, i.e., $g$ is the only edge connecting vertices in $\{x,y\}$
with vertices $\{u,v\}$. Hence, $d_G(x,u)=1$ and $d_G(x,v), d_G(y,u), d_G(y,v) \neq 1$. Since
$\diam_G \leq 2$, we have $d_G(x,v) = d_G(y,u) = d_G(y,v) =2$. Hence, $d_G(x,u) + d_G(y,v) \neq
d_G(x,v) = d_G(y,u)$ and thus, $(e,f)\not\in \ovTheta_G$, If $\ell=2$, then paw-freeness of $G$
implies that the edge $\{y,v\}$ must exist. Hence, $x,y,u,v$ induce a cycle $C_4$ in which $e$ and
$f$ are antipodal edges. By Obs.\ \ref{obs:antipodal}, $(e,f)\in \Theta_G$ and thus, $(e,f)\not\in
\ovTheta_G$. Finally, assume that $\ell=3$. In this case, we can w.l.o.g.\ assume that the edges
$\{x,v\}$ and $\{y,v\}$ exists. Hence, $d_G(x,u) = d_G(x,v) = d_G(y,v) =1$ and
$d_G(y,u)=2$. Again, $d_G(x,u) + d_G(y,v) \neq d_G(x,v) = d_G(y,u)$ and thus, $(e,f)\not\in
\ovTheta_G$. In summary, $\diam_G \leq 2$ and $\{K_4,2K_2,\text{paw}\}$-freeness of $G$ imply that
$(e,f)\not\in \ovTheta_G$ for all non-adjacent edges $e$ and $f$ of $G$, i.e., Statement (4) implies
(1) which completes this proof.
\end{proof}

As the next result shows, it is, in some cases, enough to check the containment of two edges $e$ and
$f$ in some subgraph of $G$ to determine whether $(e,f) \in \ovTheta_G^*$ or $(e,f) \in \Theta_G^*$.

\begin{lemma}\label{lm:small}
Let $G$ be a graph and let $H$ be an induced subgraph of $G$. If $H$ is a paw, a $C_4$, a gem, a
$K_5$, an $S_n$ or an $X$-house, then $(e,f) \in \ovTheta_G^*$ for all edges $e,f$ of $H$. Moreover,
if $H$ is a diamond, a $K_n$ with $n\geq 3$ or a $K_{2,3}$ then $(e,f) \in \Theta_G^*$ for all edges
$e,f$ of $H$. 
\end{lemma}
\begin{proof}
For all of the considered subgraphs $H$ of $G$ it holds that $\diam_H\leq 2$. Since $H$ is induced,
Obs. \ref{obs:diam2iso} shows that $H$ is isometric. Hence, $(e,f) \in \ovTheta_G^*$ for all edges
$e,f$ of $H$ for which $(e,f) \in \ovTheta_H^*$ is satisfied. In other words, it suffices to show
$(e,f) \in \ovTheta_H^*$ for all edges $e,f$ of $H$, that is, $\ovTheta_H^*$ is $1$-trivial. If $H$
is an $S_n$ or a paw, then $H$ contains a cut-edge and Lemma \ref{lem:cut-edge} implies that
$\ovTheta_H^*$ is $1$-trivial. The respective graph representations of $\ovTheta_H$ of all other
stated graphs $H$ are shown in Fig.\ \ref{fig:Thbg}. Since all these graph representations are
connected graphs, it immediately follows that $\ovTheta_H^*$ is an equivalence relations that has
only one equivalence class and is, thus, $1$-trivial. 
	
If $H\cong K_{2,3}$ or if $H$ is a diamond then, by similar arguments and by referring to Fig.\
\ref{fig:Thgr}, $\Theta_H^*$ is $1$-trivial. Finally assume that $H\cong K_n$, for $n\geq 3$. Any
two distinct adjacent edges are in relation $\Theta_H$ since they are contained in a common $K_3$
and by Lemma \ref{lem:furtherBasics}. Suppose that $e=\{u,v\}$ and $f=\{x,y\}$ are non-adjacent
edges in $H$. Then $(\{u,v\},\{v,x\}),(\{v,x\},\{x,y\})\in \Theta_H$, which implies
$(\{u,v\},\{x,y\})\in \Theta_H^*$. Hence $(e,f)\in \Theta_H^*$ for all edges $e$ and $f$ and thus,
$\Theta_H^*$ is $1$-trivial.
\end{proof}

\begin{figure}[t]
\centering
\includegraphics[scale=0.7]{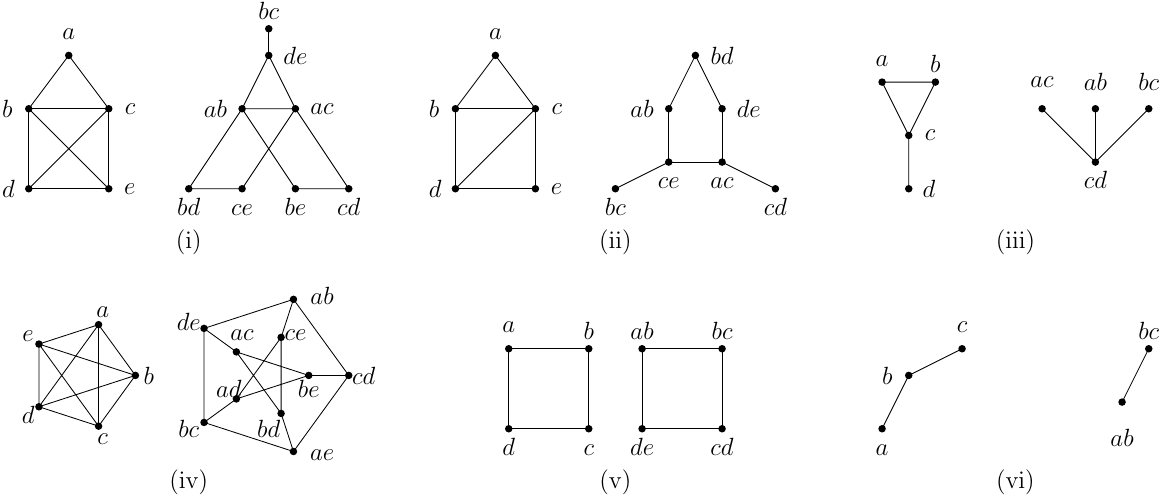}
\caption{In each of the Panels (i) to (vi), left a graph $G$ and right the graph representation
         $\mathscr G_{\ovTheta_G}$ of $\ovTheta_G$ is shown. In particular, the graph $G$ is an
         $X$-house in (i), a gem in (ii), a paw in (iii), a $K_5$ in (iv), a $C_4$ in (v), and an
         $S_2$ in (vi). In the graph representations of $\ovTheta_G$ we have used ``$xy$'' to denote
         the edge $\{x,y\}$. In all cases, $\mathscr G_{\ovTheta_G}$ is connected and, therefore,
         $\ovTheta_G^*$ is $1$-trivial.}
\label{fig:Thbg}
\end{figure}

\begin{figure}[t]
\centering
\includegraphics[scale=0.7]{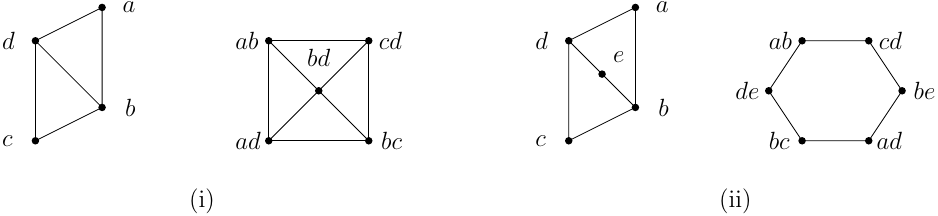}
\caption{In both Panels (i) and (ii), left a graph $G$ and right the graph representation $\mathscr
         G_{\Theta_G}$ of $\Theta_G$ is shown. In particular, the graph $G$ is a diamond in (i) and
         a $K_{2,3}$ in (ii). In the graph representations of $\Theta_G$ we have used ``$xy$'' to
         denote the edge $\{x,y\}$. In all cases, $\mathscr G_{\Theta_G}$ is connected and,
         therefore, $\Theta_G^*$ is $1$-trivial.}
\label{fig:Thgr}
\end{figure}

Somewhat surprisingly, $(e,f) \in \ovTheta_G^*$ holds for distinct edges $e,f$ in a graph $G$
whenever $e$ and $f$ are not contained in one of two small induced subgraphs on at most 4 vertices,
i.e., it is, in this case, not needed to compute distances of the vertices in $e$ and $f$ to determine
as whether $e$ and $f$ are in relation $\ovTheta^*$.

\begin{lemma}\label{lem:caseC}
Let $G$ be a graph and let $e,f$ be two distinct edges of $G$. Then it holds that $(e,f) \in
\ovTheta_G^*$ whenever $e$ and $f$ are not contained in a common $K_3$ and in a common induced
diamond in $G$.
\end{lemma} 
\begin{proof}
Put $\ovTheta\coloneqq \ovTheta_G$. Let $e=\{x,y\}$ and $f=\{u,v\}$ be two distinct edges of $G$
that are neither contained in a common $K_3$ nor in a common induced diamond. If $e$ and $f$ are
contained in distinct connected components of $G$, then trivially $(e,f)\in \ovTheta\subseteq
\ovTheta^*$. Hence, suppose that $e$ and $f$ are contained in the same connected component of $G$.
Without loss of generality, we may assume that $d(x,u)=k\coloneqq \min
\{d(x,u),d(x,v),d(y,u),d(y,v)\}$. While $x=u$ may hold, minimality of $d(x,u)$ and $e\neq f$ implies
that $y \neq v$ always holds. Suppose first that $k=0$, that is, $x=u$. Since $e$ and $f$ are not
edges of an induced $K_3$ of $G$, the vertices $x,y,v$ induce a shortest $yv$-path. This together
with Lemma \ref{lem:SP} implies that $(e,f) \in \ovTheta \subseteq \ovTheta^*$. 

Suppose now that $k \geq 1$. In particular, we have $x \neq u$, that is, the vertices $x,y,u$ and
$v$ are pairwise distinct. Let $P$ be a shortest path in $G$ between $x$ and $u$ and let $x'$,
resp., $u'$ be the vertex adjacent to $x$, resp., $u$ in $P$. Put $e_x=\{x,x'\}$ and $f_u=\{u,u'\}$.
Note that if $k=1$, we have $x=u'\neq x'=u$ and if $k=2$, we have $x'=u'$. If otherwise, $k \geq 3$,
then $x,x',u',u$ are pairwise distinct. Since $P$ is a shortest path between $x$ and $u$ in $G$, $P$
is an induced path of $G$. Hence, by Lemma~\ref{lem:path}, we have $(e',f') \in \ovTheta^*$ for all
edges $e',f'$ of $P$. In particular, $(e_x,f_u) \in \ovTheta^*$. Note that there cannot be a vertex
$z \in V(P) \setminus \{x,x'\}$ such that $\{y,z\}$ is an edge of $G$ since then $d(y,u) < d(x,u)$;
a contradiction to the minimality of $d(x,u)$. Similarly, there is no vertex $z \in V(P) \setminus
\{u,u'\}$ such that $\{v,z\}$ is an edge of $G$. Now, three cases may occur for the
adjacency-relation between $y$ and $x'$ together with $v$ and $u'$

\begin{owndesc}
\item[\textnormal{\em Case (1): None of $\{y,x'\}$ and $\{u',v\}$ are edges of $G$.}] 
	In this case, both $e$ and $e_x$  as well as $f$ and $f_u$ form an induced 
	$P_3$ and Lemma~\ref{lem:path} implies that $(e,e_x) \in \ovTheta^*$
	and $(f,f_u) \in \ovTheta^*$. As argued above, we have $(e_x,f_u) \in \ovTheta^*$.
	Taken the latter two arguments together, we have $(e,f) \in \ovTheta^*$.

\item[\textnormal{\em Case (2): Exactly one of $\{y,x'\}$ or $\{u',v\}$ is an edge of $G$.}] 
	We may assume w.l.o.g.\ that $\{y,x'\}\in E(G)$. Suppose first that $k=1$.
	Hence, $x'=u$ and $u'=x$ and, in particular, $G$ contains the edges $e, f,
	\{x,u\}$ and $\{y,u\}$. By assumption, $\{v,x\}$ is not an edge of $G$. Hence,
	$x,u,y,v$ induce either a paw or a diamond. Since by assumption, $e$ and $f$ are not contained
	in a common induced diamond in $G$, these four vertices induce a paw. 
	By Lemma~\ref{lm:small}, $(e,f) \in \ovTheta^*$ follows.
				
	Assume now that $k>1$. Hence, $x' \neq u$. As in Case (1), we have $(f,f_u) \in \ovTheta^*$.
	Moreover, for $x''$ the neighbor of $x'$ in $P$ distinct from $x$ (which must exist since $x'
	\neq u$), we have $(\{x',x''\},f_u) \in \ovTheta^*$, since $(e',f') \in \ovTheta^*$ holds for all
	edges $e',f'$ of $P$. Moreover, since neither $\{x,x''\}$ nor $\{y,x''\}$ are edges in $G$,
	$x,y,x'x''$ must induce a paw and, therefore, $(\{x',x''\},e) \in \ovTheta^*$ holds by
	Lemma~\ref{lm:small}. Since $(\{x',x''\},f_u) \in \ovTheta^*$ it follows that $(e,f_u)\in
	\ovTheta^*$ which together with $(f,f_u) \in \ovTheta^*$ implies that $(e,f)\in \ovTheta^*$. 

\item[\textnormal{\em Case (3): Both $\{y,x'\}$ and $\{u',v\}$ are edges in $G$.}] 
	Assume first that $k=1$. Hence, $x'=u$ and $u'=x$ and one easily observes that $x,y,u,v$ induce
	either a $K_4$ or a diamond. Since $e$ and $f$ cannot be contained in a common induced diamond,
	it follows that $x,y,u,v$ induce a $K_4$. Hence, $\Delta_{ef}=\{1\}$ and contraposition of Lemma
	\ref{lem:theta-delta} implies that $(e,f) \in \ovTheta \subseteq \ovTheta^*$.

	Assume now that $k>1$. By similar arguments as in Case (2), 
	$x,x',x'',y$ as well as $y,u,u',u''$ induced a paw and we have 
	$(\{x',x''\},e) \in \ovTheta^*$ and $(\{u',u''\},f) \in \ovTheta^*$ by Lemma~\ref{lm:small}. 
	Since $(e',f') \in \ovTheta^*$ holds for all edges $e',f'$ of $P$, it holds in particular 
	 $(\{x',x''\}, \{u',u''\})\in \ovTheta^*$. 
	Taken the latter arguments together, $(e,f) \in \ovTheta^*$.
\end{owndesc}
In summary, in all of the three cases we have $(e,f) \in \ovTheta^*$ which completes the proof.
\end{proof}

Clearly, the converse of Lemma \ref{lem:caseC} is not always
satisfied. As an example consider an induced paw in $G$ that contains 
edges $e$ and $f$ that are located in common $K_3$ and satisfy, by Lemma \ref{lm:small},
$(e,f)\in \ovTheta^*$. 

\begin{proposition}\label{prop_K3-free}
For every $K_3$-free graph $G$ it holds that $\ovTheta^*_G$ is $1$-trivial.
\end{proposition}
\begin{proof}
If $G$ is $K_3$-free, it is, in particular, diamond-free. Hence, none of the edges of $G$ are
contained in a common $K_3$ or diamond. By Lemma \ref{lem:caseC}, $(e,f)\in \ovTheta^*_G$ for all
edges $e,f$ in $G$. Therefore, $\ovTheta^*_G$ is $1$-trivial.
\end{proof}

The next result strengthens Prop.~\ref{prop_K3-free} and provides a characterization of graphs $G$
with $1$-trivial $\ovTheta_G^*$ in terms of small forbidden and enforced subgraphs.

\begin{proposition}\label{prop:ovTheta*-trivial}
The following statements are equivalent for every connected graph $G$.
\begin{enumerate}
\item  $\ovTheta_G^*$ is $1$-trivial.
\item (I) $|V(G)| \leq 4$ and $G$ is not a $K_3$,  not a diamond and not a $K_4$

      (II) $|V(G)| \geq 5$ and $G$ is $K_3$-free or contains at least one the induced subgraphs: a $K_5$, a paw, a $P_4$
\end{enumerate}
\end{proposition}
\begin{proof}
In what follows, put $\ovTheta\coloneqq \ovTheta_G$ and let $G=(V,E)$. We start with the assumption
that $|V(G)|\leq 4$. By contraposition, if $G$ is a $K_3$, then none of any two distinct edges are
in relation $\ovTheta$ (cf.\ Lemma \ref{lem:furtherBasics}) and hence $\ovTheta^*$ has three
equivalence classes. If $G$ is a diamond (resp., a $K_4$), then it is easy to see that $\ovTheta^*$
consists of three equivalence classes. In all cases $\ovTheta^*$ is not $1$-trivial. For the
converse, suppose that $G$ is not a $K_3$, not a diamond and not a $K_4$. Clearly, if $|V(G)|\in
\{1,2\}$, $\ovTheta^*$ is $1$-trivial. Moreover, if $|V(G)| = 3$, then $G\simeq P_3$ and
$\ovTheta^*$ is $1$-trivial. If $|V(G)| = 4$, then $G$ must be a $P_4$, a $C_4$, an $S_3$ or a paw.
In the latter three cases, Lemma \ref{lm:small} implies that $\ovTheta^*$ is $1$-trivial. If
$G\simeq P_4$, then Lemma \ref{lem:path} implies that $\ovTheta^*$ is $1$-trivial. In summary,
Statement (1) and (2) are equivalent whenever $|V(G)|\leq 4$. 

Hence, assume in the following that $|V|\geq 5$. We show first that (2.II) implies (1). If $G$ is
$K_3$-free, then $G$ is also diamond-free. This together with Lemma~\ref{lem:caseC} implies that
$(e,f) \in \ovTheta^*$ for all distinct edges $e$ and $f$ of $G$. 

Assume now that $H\simeq K_5$ is a subgraph of $G$. By Lemma \ref{lm:small}, $(e,f) \in \ovTheta^*$
for all edges $e,f$ of $H$. Hence, if $|V|=5$, then we are done. Suppose that $|V|>5$. In this case,
there is an edge $e=\{x,y\}$ that is not contained in $H$. Moreover, $H$ contains an edge
$f=\{u,v\}$ that is not adjacent to $e$. We next show that $(e,f) \in \ovTheta^*$ holds. Since $e$
and $f$ are non-adjacent, $e$ and $f$ are not located in a common $K_3$. This together with
Lemma~\ref{lem:caseC} implies that, if there is no induced diamond of $G$ that contains both $e$ and
$f$, then $(e,f) \in \ovTheta^*$ holds as desired. Hence, suppose that $e$ and $f$ are contained in
a common induced diamond $H'$. Since $e$ and $f$ are non-adjacent, we can assume w.l.o.g.\ that $H'$
has edges $\{x,y\}$, $\{u,v\}$, $\{x,u\}$, $\{y,u\}$ and $\{y,v\}$. Now, pick a vertex $w$ in $V(H)$
distinct from $u$ and $v$. If neither $\{x,w\}$ nor $\{y,w\}$ are edges of $G$, then $\{x,y,u,v,w\}$
induces a gem in $G$. If $\{y,w\}$ is an edge of $G$ and $\{x,w\}$ is not, then, $\{x,y,u,v,w\}$
induces a $X$-house in $G$. In both these cases, $(e,f) \in \ovTheta^*$ follows from
Lemma~\ref{lm:small}. If $\{x,w\}$ is an edge of $G$ and $\{y,w\}$ is not, then $\{x,y,v,w\}$
induces a $C_4$ in $G$. By Lemma \ref{lm:small}, $(e,\{v,w\}) \in \ovTheta^*$. Since $\{v,w\}$ and
$f$ are edges of $H$, it holds that $(\{v,w\},f) \in \ovTheta^*$ and, therefore, $(e,f) \in
\ovTheta^*$. Finally, if both $\{x,w\}$ and $\{y,w\}$ are edges of $G$, then $\{x,y,u,w\}$ induces a
$K_4$ in $G$ and $\Delta_{e\{u,w\}}=\{1\}$ holds. Thus, $(e,\{u,w\}) \in \ovTheta\subseteq
\ovTheta^*$. Again, since $\{u,w\}$ and $f$ are edges of $H$, we have $(\{u,w\},f) \in \ovTheta^*$
and, therefore, $(e,f) \in \ovTheta^*$.

Assume now that $G$ contains a paw or $P_4$, call it $H$, induced by $a,b,c,d \in V$. Independent of
whether $H$ is a paw or a $P_4$, $H$ contains a vertex that has degree $1$ in $H$. We may assume
w.l.o.g.\ that $a$ has degree $1$ in $H$, and that $\{a,b\}$ and $\{b,c\}$ are edges of $H$. Note
that the latter uniquely determines the remaining edges of $H$. If $H$ is a paw (resp., a $P_4$),
then Lemma~\ref{lm:small} (resp., Lemma~\ref{lem:path}) imply that $(e',f') \in \ovTheta^*$ for all
edges $e',f'$ of $H$. Now, let $e=\{x,y\}$ be an edge of $G$ that is not an edge of $H$, and put
$f\coloneqq\{a,b\}$. We now show that $(e,f) \in \ovTheta^*$ must hold. If $e$ and $f$ are not part of the
same induced $K_3$ or diamond of $G$, then $(e,f) \in \ovTheta^*$ holds by Lemma~\ref{lem:caseC}.
Otherwise, one of the following must hold: $e$ and $f$ are edges of (A) an induced $K_3$ or (B) an
induced diamond in $G$.

\smallskip
\noindent
\emph{Case (A): $e$ and $f$ are edges of an induced $K_3$.} In this case, $|\{a,b\}\cap\{x,y\}|=1$
                and $\{c,d\}\cap\{x,y\}= \emptyset$. W.l.o.g.\ we may assume that $y=a$ or $y=b$
                holds. Thus, $e$ is one of $\{x,a\}$ or $\{x,b\}$, i.e., $e\in F =
                \{\{x,a\},\{x,b\}\}$ and the $K_3$ is induced by $a,b,x$. We now distinguish between
                the cases that $x$ and $c$ are or are not adjacent.
					 
\begin{itemize}
	\item[(i)] $\{x,c\}$ is not an edge of $G$. Then, $\{x,a,b,c\}$ induces a paw in $G$ hat contains $e$ and $f$. By
	           Lemma~\ref{lm:small}, we have $(e,f) \in \ovTheta^*$.

	\item[(ii)] $\{x,c\}$ is an edge of $G$. Then four subcases may occur. 

	\begin{itemize}
		\item[(a)] $H$ is a $P_4$ and $\{x,d\}$ is not an edge of $G$. Then $\{a,x,c,d\}$ induces a
		           $P_4$ in $G$, so $(\{x,a\},\{c,d\}) \in \ovTheta^*$ by Lemma~\ref{lem:path}, and
		           $\{x,b,c,d\}$ induces a paw in $G$, so $(\{x,b\},\{c,d\}) \in \ovTheta^*$ by
		           Lemma~\ref{lm:small}. Since $(\{c,d\},f) \in \ovTheta^*$ and $e\in F$,
		            $(e,f) \in \ovTheta^*$ follows.

		\item[(b)] $H$ is a $P_4$ and $\{x,d\}$ is an edge of $G$. Then $\{x,a,b,c,d\}$ induces a gem
		           in $G$ that contains $e$ and $f$.  By Lemma~\ref{lm:small},  $(e,f) \in \ovTheta^*$.

		\item[(c)] $H$ is a paw and $\{x,d\}$ is not an edge of $G$. Then $\{x,a,b,c,d\}$ induces a
		           gem in $G$ that contains $e$ and $f$ and, again,  Lemma~\ref{lm:small} implies  $(e,f) \in \ovTheta^*$.
			      
		\item[(d)] $H$ is a paw and $\{x,d\}$ is an edge of $G$. Then $\{x,a,b,c,d\}$ induces a
		           $X$-house in $G$ that contains $e$ and $f$.  By Lemma~\ref{lm:small},  $(e,f) \in \ovTheta^*$.
	\end{itemize}
\end{itemize}

\noindent
\emph{Case (B): $e$ and $f$ are edges of an induced diamond $D$.}
Note that if $e$ and $f$ share a vertex, then they are either edges of an induced $K_3$, in which
case we are back to case (A), or they are not part of the same induced $K_3$, in which case $e\cup
f$ induce a path on three vertices and we have $(e,f)\in \ovTheta^*$
by Lemma \ref{lem:path}. Hence, it remains to investigate the case where $e$ and $f$ are
non-adjacent edges of an induced diamond.

By assumption $e$ and $f$  as well as $f$ and $\{c,d\}$ are non-adjacent. However, 
$e=\{c,d\}$ is not possible since then $e\cup f = \{a,b,c,d\}$ would not induce a diamond
as they induce a $P_4$ or a paw. We thus, distinguish between the case 
$|e  \cap \{c,d\}|=1$ and $e  \cap \{c,d\} \neq \emptyset$. 
We start with the first case and  assume w.l.o.g.\ that $x \in \{c,d\}$
and thus, $y \notin \{a,b,c,d\}$.

Suppose first that $x=c$. Since $\{a,b,x,y\}$ induces a diamond in $G$ and $x=c$ and $a$ are not adjacent in $G$,
both $\{a,y\}$ and $\{b,y\}$ are edges of $G$. From there, three cases may occur:
\begin{itemize}
\item[(i)] $H$ is a paw. Then $\{a,b,c,d,y\}$ induces a $X$-house in $G$ if $\{y,d\}$ is an edge of
           $G$, and a gem otherwise. In both cases, $(e,f) \in \ovTheta_G^*$ follows from
           Lemma~\ref{lm:small}.
\item[(ii)] $H$ is a $P_4$ and $\{y,d\}$ is an edge of $G$. Then $\{a,b,c,d,y\}$ induces a gem in
            $G$, and $(e,f) \in \ovTheta_G^*$ follows from Lemma~\ref{lm:small}.
\item[(iii)] $H$ is a $P_4$ and $\{y,d\}$ is not an edge of $G$. Then $\{b,c,d,y\}$ induces a paw in
             $G$, so $(e,\{c,d\}) \in \ovTheta_G^*$ by Lemma~\ref{lm:small}. Since $(f,\{c,d\}) \in
             \ovTheta_G^*$, $(e,f) \in \ovTheta_G^*$ follows.
\end{itemize}

Suppose now that $x=d$. If $H$ is a $P_4$, then $x$ is neither adjacent to $a$ nor to $b$. But then
$\{a,b,x,y\}$ cannot induce a diamond in $G$. Hence, $H$ must be a paw. In this case, we have the
edges $\{a,y\}$ and $\{b,y\}$ in $G$ since $x$ and $a$ are not adjacent. Then, $\{a,b,c,d,y\}$
induces an $X$-house in $G$ if $\{y,c\}$ is an edge of $G$, and a gem otherwise. In both cases,
$(e,f) \in \ovTheta_G^*$ follows from Lemma~\ref{lm:small}.

Now, consider the case $e \cap \{c,d\}=\emptyset$. Without loss of generality, we may assume that $x$ has
degree $2$ in $D$ and $y$ has degree $3$ in $D$. In particular, $\{y,a\}$ and $\{y,b\}$ are both
edges of $H$, and exactly one of $\{x,a\}$ or $\{x,b\}$ is an edge of $G$. We remark first that, if
neither $\{x,c\}$ nor $\{y,c\}$ are edges of $G$, then $\{x,y,b,c\}$ induces a paw in case $\{x,b\}$
is an edge of $G$ and a $P_4$ otherwise. It follows from Lemma~\ref{lm:small} and
Lemma~\ref{lem:path} that $(e,\{b,c\}) \in \ovTheta^*$ must hold. Since, in addition, $(\{b,c\},f)
\in \ovTheta^*$ holds, it follows that $(e,f) \in \ovTheta^*$. Next, we consider the cases where $c$
is adjacent to at least one of $x$ or $y$.

\begin{itemize}
\item[(i)] 
	$\{x,c\}$ is an edge of $G$ and $\{y,c\}$ is not. If $\{x,b\}$ is an edge of $G$, then
	$\{x,y,a,b,c\}$ induces a gem in $G$, and $(e,f) \in \ovTheta^*$ follows from
	Lemma~\ref{lm:small}. If otherwise, $\{x,a\}$ is an edge of $G$, then $\{x,y,b,c\}$ induces a
	$C_4$ in $G$. In that case, we have $(e,\{b,c\}) \in \ovTheta^*$ by Lemma~\ref{lm:small} which
	together with $(\{b,c\},f) \in \ovTheta^*$ implies that $(e,f) \in \ovTheta^*$.

\item[(ii)] 
	$\{y,c\}$ is an edge of $G$ and $\{x,c\}$ is not. If $\{x,a\}$ is an edge of $G$, then
	$\{x,y,a,b,c\}$ induces a gem in $G$, and $(e,f) \in \ovTheta^*$ follows from
	Lemma~\ref{lm:small}. Suppose now that $\{x,a\}$ is not an edge of $G$, that is, $\{x,b\}$ is an
	edge of $G$. In that case, we need to consider the vertex $d$. If neither $\{x,d\}$ nor $\{y,d\}$
	are edges of $G$, then $\{x,y,c,d\}$ induces a $P_4$ in $G$. If $\{x,d\}$ is an edge of $G$ and
	$\{y,d\}$ is not, then $\{x,y,c,d\}$ induces a $C_4$ in $G$. If $\{y,d\}$ is an edge of $G$ and
	$\{x,d\}$ is not, then $\{x,y,c,d\}$ induces a paw in $G$. In all these cases,
	Lemma~\ref{lem:path} and Lemma~\ref{lm:small} imply that $(e,\{c,d\}) \in \ovTheta^*$ . Since
	$(\{c,d\},f) \in \ovTheta^*$ also holds, $(e,f) \in \ovTheta^*$ follows. Finally, if both
	$\{x,d\}$ and $\{y,d\}$ are edges of $G$, then $\{x,y,a,b,d\}$ induces a $X$-house in $G$ in case
	$H$ is a paw and a gem if $H\simeq P_4$. 
	In both cases, $(e,f) \in \ovTheta^*$ follows from Lemma~\ref{lm:small}.

\item[(iii)] 
	Both $\{x,c\}$ and $\{y,c\}$ are edges of $G$. If $\{x,b\}$ is an edge of $G$, then
	$\{x,y,a,b,c\}$ induces a $X$-house in $G$, and $(e,f) \in \ovTheta^*$ follows from
	Lemma~\ref{lm:small}. Suppose now that $\{x,b\}$ is not an edge of $G$, that is, $\{x,a\}$ is an
	edge of $G$. In that case, we need to consider the vertex $d$. If neither $\{x,d\}$ nor $\{y,d\}$
	are edges of $G$, then $\{x,y,c,d\}$ induces a paw in $G$. If exactly one of $\{x,d\}$ or
	$\{y,d\}$ is an edge of $G$, then $\{x,y,a,c,d\}$ induces a gem in $G$. In both these cases
	Lemma~\ref{lm:small} implies that $(e,\{c,d\}) \in \ovTheta^*$ . Since $(\{c,d\},f) \in
	\ovTheta^*$ also holds, $(e,f) \in \ovTheta^*$ follows. Finally, if both $\{x,d\}$ and $\{y,d\}$
	are edges of $G$, then $\{x,y,a,c,d\}$ induces an $X$-house in $G$, so $(e,\{c,d\}) \in \ovTheta$
	follows again from Lemma~\ref{lm:small}. Again, $(e,f) \in \ovTheta^*$ follows from the fact that
	$(\{c,d\},f) \in \ovTheta^*$.
\end{itemize}
In summary, if $|V| \geq 5$ and $G$ is $K_3$-free or contains a $K_5$, a paw, a $P_4$ as an induced subgraphs: 
then  $\ovTheta_G^*$ is $1$-trivial. \smallskip

We show now, by contraposition, that (1) implies (2.II) for the case $|V| \geq 5$. Hence, suppose
(2.II) is not satisfied, i.e., $G$ is not $K_3$-free and does not contain a $K_5$, a paw, a $P_4$ as
an induced subgraph. Since $G$ is not $K_3$-free, there are three vertices $x,y,z \in V$ that induce
a $K_3$ in $G$. Put $e\coloneqq \{x,y\}$ and $f\coloneqq \{x,z\}$. By Lemma \ref{lem:furtherBasics},
$(e,f)\notin \ovTheta$. We continue with showing that $(e,f)\notin \ovTheta^*$ which immediately
implies that $\ovTheta_G^*$ is not $1$-trivial. 

Assume, for contradiction, that $(e,f)\in \ovTheta^*$. Hence, there are edges $e_1,\dots, e_\ell\in
E$, $\ell\geq 1$ such that $(e,e_1), (f,e_\ell),(e_i,e_{i+1})\in \ovTheta$,
$1\leq i<\ell$. Without loss of generality, we may assume that $e$ and $f$ are chosen such that
$\ell$ is minimum, i.e., there are no two edges $e',f'$ that are part of a common $K_3$ (and thus,
$(e',f')\notin \ovTheta$) for which there are $\ell' < \ell$ edges edges $e'_1,\dots,e'_{\ell'}$
with $(e',e'_1), (f',e'_\ell),(e'_i,e'_{i+1})\in \ovTheta$, $1\leq i<\ell'$ (implying $(e',f')\in
\ovTheta^*$). This, in particular, implies that no two edges in $\{e, e_1, \ldots, e_\ell,f\}$ are
part of a common $K_3$ in $G$, except the edges $e$ and $f$. 

Consider now $e_1=\{u,v\}$. Since $(e,e_1) \in \ovTheta$, we have 
\begin{equation}
d(x,u)+d(y,v)=d(x,v)+d(y,u). \label{eq:1}
\end{equation}
Note that $K\coloneqq \max\{d(x,u),d(y,v),d(x,v),d(y,u)\} \leq 2$ since $G$ does not contain an
induced $P_4$. We continue with proving the following claim. 

\begin{owndesc}
	\item[{\textnormal{\em Claim X:  $\{x,y,z,u,v\}$ induces a $K_5$ from which either the edge $\{z,v\}$ or $\{z,u\}$ has been removed.}}] \ \\
	W.l.o.g.\ we may assume that $d(x,u)=\min\{d(x,u),d(y,v),d(x,v),d(y,u)\}$. Put $k\coloneqq
	d(x,u)$. Since $k\leq 2$, we continue to consider the following subcases $k=0$, $k=1$, and $k=2$.

\begin{owndesc}
\item[\textnormal{\em Case $k=0$:}] 
	In this case, we have $x=u$ and thus, $d(x,v)=1$ and $d(y,u)=1$. This together with Equ.\
	\eqref{eq:1} and $d(x,u)=0$ implies that $d(y,v)=2$. Hence, $\{v,y\}$ is not an edge of $G$. If
	$\{v,z\}$ is not an edge of $G$, then $\{x,y,z,v\}$ induces a paw in $G$; a contradiction to
	paw-freeness of $G$. If otherwise, $\{v,z\}$ is an edge of $G$, then $\{u=x,v,z\}$ induces a
	$K_3$ in $G$. In particular, $e_1=\{u,v\}$ and $f=\{x,z\}$ are part of the same $K_3$, which
	contradicts the choice of $e$ and $f$ and minimality of $\ell$. Hence, $k=0$ cannot hold.

\item[\textnormal{\em Case $k=1$:}] 
	In this case, $\{x,u\}$ is an edge of $G$ and thus, $d(x,u)=1$. Furthermore,
	$d(y,u),d(y,v),d(x,v)\geq k$ implies that $u\neq x,y$ and $v\neq x,y$. Suppose first that
	$\{y,u\}$ is not an edge in $G$ and thus, $u\neq z$ and $d(y,u)=2$. Equ.\ \eqref{eq:1} and the
	fact that $k\leq 2$ imply that $d(y,v)=2$ and $d(x,v)=1$ must hold. Hence, $\{x,v\}$ is an edges
	of $G$ while $\{y,v\}$ is not. Thus, $\{x,y,u,v\}$ induces a paw in $G$; a contradiction to
	paw-freeness of $G$.

	Therefore, $\{y,u\}$ must be an edge in $G$. Equ.\ \eqref{eq:1} together with $d(x,u)=d(y,u)=1$
	implies that $d(x,v)=d(y,v)\leq 2$ and thus, either both or none of $\{x,v\}$ and $\{y,v\}$
	are edges of $G$. Note that $z=u$ or $z=v$ is possible. If none of $\{x,v\}$ and $\{y,v\}$ are
	edges of $G$, then $z\neq u,v$ which together with $u,v\neq x,y$ implies that $\{x,y,u,v\}$
	induces a paw in $G$; a contradiction to paw-freeness of $G$.
	
	Therefore, both of $\{x,v\}$ and $\{y,v\}$ are edges of $G$. Since $u,v,x,y$ are pairwise
	distinct, $\{x,y,u,v\}$ induces a $K_4$ in $G$. We distinguish now between the cases $z=u$, $z=v$
	and $z\neq u,v$. Assume that $z=u$. Since $\{x,u\}$, $\{u,v\}$ and $\{x,v\}$ are edges in $G$, it
	follows that $e_1=\{u,v\}$ and $f=\{x,z\}$ are part of the same $K_3$, which contradicts the
	choice of $e$ and $f$ and minimality of $\ell$. Hence, $z\neq u$ holds. By similar arguments,
	$z\neq v$ holds. Thus, $x,y,z,u,v$ are pairwise distinct. If neither $\{z,u\}$ not $\{z,v\}$ are
	edges of $G$, then $\{x,z,u,v\}$ induces a paw in $G$; a contradiction to paw-freeness of $G$. If
	both $\{z,u\}$ and $\{z,v\}$ are edges of $G$, then $\{x,y,z,u,v\}$ induces a $K_5$ in $G$; a
	contradiction to $K_5$-freeness of $G$. Hence exactly one of $\{z,u\}$ or $\{z,v\}$ is an edge of
	$G$ and \emph{Claim X} holds.

\item[\textnormal{\em Case $k=2$:}] 
	Finally, suppose that $k=2$. Then, we must have $d(x,u)=d(y,v)=d(x,u)=d(y,v)=2$. In particular,
	neither $\{x,u\}$ nor $\{x,v\}$ are edges of $G$, and there exists $w \in G$ such that $\{x,w\}$ and
	$\{u,w\}$ are edges of $G$ (note that $w=z$ might hold). However, $\{x,u,v,w\}$ induces a paw
	in case $\{v,w\}$ is an edge of $G$, and a $P_4$ otherwise. None of the cases can occur since 
	$G$ is $P_4$- and paw-free. This completes the proof of \emph{Claim X}. 
\end{owndesc}
\end{owndesc}

We can apply similar arguments to the edge $e_\ell=\{u',v'\}$ to conclude that $\{x,y,z,u',v'\}$
induces a $K_5$ from which either the edge $\{y,v'\}$ or $\{y,u'\}$ has been removed. Note that
$e_1\cap e_\ell \neq \emptyset$ may be possible. However, since we can choose the vertex labelings
in $e_1=\{u,v\}$ and $e_\ell=\{u',v'\}$ arbitrarily, we can assume that $\{z,v\}$ and $\{y,v'\}$ are
not edges in $G$. This and the fact that $\{y,v\}$ is an edge of $G$ implies that $v\neq v'$.
Moreover, $\{v,v'\}$ must be an edge of $G$, as otherwise, $\{y,z,v,v'\}$ induces a $P_4$ in $G$.
Note that $u\neq v$ and $u'\neq v'$ always holds. In addition, since $\{y,u\}$ is an edge of $G$ but
not $\{y,v'\}$, we have $u\neq v'$. Similarly, $\{z,u'\}\in E(G)$ and $\{z,v\}\notin E(G)$
implies that $u'\neq v$. If, however, $u=u'$ holds, then $\{u=u',v,v'\}$ induces a $K_3$ that
contains $e_1=\{u,v\}$ and $e_\ell=\{u',v'\}$; a contradiction to the choice of $e$ and $f$. Hence,
$u \neq u'$ must hold. In summary, the vertices $x,y,z,u,v,u',v'$ are pairwise distinct.
	
Note that $\{u,u'\}$ cannot be an edge of $G$ since, otherwise, $\{x,y,z,u,u'\}$ induces a $K_5$ in
$G$, which is forbidden. Moreover, $\{u',v\}$ must be an edge of $G$, since otherwise,
$\{u,v,x,u'\}$ induces a paw in $G$. Note that $d(x,v) + d(y,u') = 1+1 = d(x,u') + d(y,v)$ and,
thus, for $e'=\{u',v\}$ we have $(e',e_1) \in \ovTheta$. In particular, the sequence $e',e_1,
\ldots, e_\ell$ is such that $(e',e_1) \in \ovTheta$, and $(e_i,e_{i+1}) \in \ovTheta^*$ for all $i
\in \{1, \ldots, \ell-1\}$. As argued above, $\{v,v'\}$ is an edge of $G$ and, consequently,
$\{u',v',v\}$ induce a $K_3$ in $G$ that contains $e'$ and $e_\ell$. However, since the sequence $e', e_1, \ldots, e_\ell$ is
strictly shorter than the sequence $e, e_1, \ldots, e_\ell, f$ we obtain a contradiction to the
choice of $e$ and $f$. 

To summarize, we have shown that if (2.II) does not hold, then we can find three vertices $x,y,z$ in
a common $K_3$ and resulting edges $e=\{x,y\}$ and $f=\{x,z\}$ such that $(e,f)\notin \ovTheta^*$.
Hence, $\ovTheta_G^*$ is not $1$-trivial. Thus, (1) implies (2.II) in case $|V|\geq 5$.
\end{proof}

We can now employ Prop.\ \ref{prop:ovTheta*-trivial} to characterize graphs $G$ for which
$\ovTheta_G^*$ is not $1$-trivial. As it turns out, this is the case for only a
surprisingly limited subclass of complete multipartite graphs.

\begin{theorem}\label{thm:ovTheta*-non-1trivial}
	The following statements are equivalent for every connected graph $G$.
	\begin{enumerate}
	\item $\ovTheta_G^*$ contains $k>1$ equivalence classes.
	\item  $G$ is a complete multipartite graph $K_{n_1,\dots, n_\ell}$ with $\ell\in \{3,4\}$.
	\end{enumerate}
	In this case,  it holds that $\Delta_{e,f}\neq 3$ and $\Delta_{e,f}\subseteq \{1,2\}$
	for all non-adjacent edges $e,f$ of $G$. 	
\end{theorem}	
\begin{proof}
Suppose that $\ovTheta_G^*$ contains $k>1$ equivalence classes. If $|V(G)| \leq 4$, then
Prop.~\ref{prop:ovTheta*-trivial} implies that $G$ is a diamond (and thus, a $K_{1,1,2}$), a $K_3$
(and thus, a $K_{1,1,1}$) or a $K_4$ (and thus, a $K_{1,1,1,1}$). In all cases, $G$ is a complete
multipartite graph $K_{n_1,\dotsm n_\ell}$ with $\ell\in \{3,4\}$. Assume that $|V(G)| \geq 5$. By
Prop.~\ref{prop:ovTheta*-trivial}, $G$ is not $K_3$-free but paw-free which implies together with
Theorem \ref{thm:paw-free} that $G$ must be a complete multipartite graph $K_{n_1,\dotsm n_\ell}$.
Since $G$ contains a $K_3$, $\ell \geq 3$ must hold. By Prop.~\ref{prop:ovTheta*-trivial}, $G$ is
$K_5$-free and thus, $\ell < 5$ holds. Hence, $\ell\in \{3,4\}$.

Assume now $G$ is a complete multipartite graph $K_{n_1,\dots, n_\ell}$ with $\ell\in \{3,4\}$. In
this case, $\ell\geq 3$ implies that $G$ contains a $K_3$. This implies that if $|V(G)| \leq 4$ then
$G$ can only be a $K_{1,1,1}$, $K_{1,1,2}$ or $K_{1,1,1,1}$ and thus, a $K_3$, diamond or a $K_4$.
In this case, Prop.~\ref{prop:ovTheta*-trivial}(2.I) implies that $\ovTheta_G^*$ has $k>1$
equivalence classes. Assume that $|V(G)| \geq 5$. Since $G$ is a complete multipartite graph, $G$
does not contain induced $P_4$s. Moreover, by Theorem \ref{thm:paw-free}, $G$ is paw-free. Since
$\ell \leq 4$, $G$ contains no induced $K_5$. As argued above, $G$ is not $K_3$-free. Taken the
latter arguments together with Prop.~\ref{prop:ovTheta*-trivial}(2.II), we can conclude that
$\ovTheta_G^*$ has $k>1$ equivalence classes. 

One easily observes that $\diam_G\leq 2$ holds for every complete multipartite graph $G$.
Prop.~\ref{prop:DN3-nonAdj} implies that $\Delta_{e,f}\neq 3$ and $\Delta_{e,f}\subseteq \{1,2\}$
for all non-adjacent edges of $G$. 	
\end{proof}

In what follows we are interested in the structure of the $\ovTheta$-relation for which $\ovTheta^*$
is not $1$-trivial. In particular, we ask when an arbitrary binary relation is such a
$\ovTheta$-relation. Clearly, based on Theorem~\ref{thm:ovTheta*-non-1trivial}, the structure of
$\ovTheta$ must be rather restricted. Nevertheless, its structure is surprisingly closely related to
Cartesian graph products, see Fig.\ \ref{fig:R=ovT}.

\begin{figure}[t]
\centering
\includegraphics[width=0.8\textwidth]{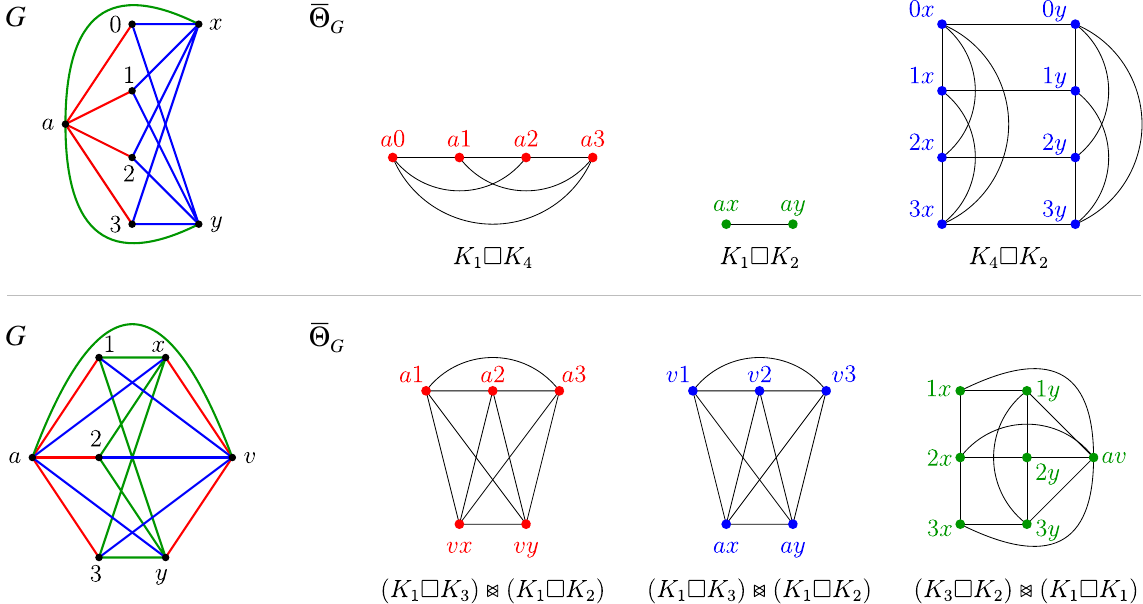}
\caption{\emph{Top:} A complete multipartite graph $G = K_{1,2,4}$ and its relation $\ovTheta_G$. 
			Here $\mathscr{G}_{\ovTheta_G}$ has three connected components $H_1,H_2$ and $H_3$ where 
			$H_1\simeq K_1\Box K_4$, 
			$H_2\simeq K_1\Box K_2$ and 
			$H_3\simeq K_4\Box K_2$. Assuming that $V(K_1)=\{a\}$,
			$V(K_2)=\{x,y\}$ and 	$V(K_4)=\{0,1,2,3\}$ we obtain 
			the coordinates of the underlying products in 
			the connected components $H_1,H_2,H_3$ that 
			are in a 1:1 correspondence with the edges in $G$. 
		 In particular, there is  1:1 correspondence between the maximal independent sets in $G$ and the factors $K_1,K_2,K_4$. \newline{}
			\emph{Bottom:} A complete multipartite graph $G = K_{1,1,2,3}$ and its relation $\ovTheta_G$. 
			Here, $\mathscr{G}_{\ovTheta_G}$ has three connected components $H_1,H_2$ and $H_3$ where 
				$H_1\simeq H_2 \simeq (K_1\Box K_3) \protect\join (K_{1}\Box K_{2})$ and 
				$H_3\simeq (K_3\Box K_2) \protect\join (K_{1}\Box K_{1})$. Again, there is a 1:1 correspondence
				between the maximal independent sets in $G$ and the four factors used to build $\mathscr{G}_{\ovTheta_G}$: 
				two $K_1$'s with vertex set $\{a\}$, resp., $\{v\}$, a $K_2$ with vertex set $\{x,y\}$ and a $K_3$
				with vertex set $\{1,2,3\}$.
}
\label{fig:R=ovT}
\end{figure}

 \begin{theorem}\label{thm:structure-ovTheta}
For a given relation $R$  there is a connected graph $G$ such that $R=\ovTheta_G$ and  
	for which $\ovTheta_G^*$ is not $1$-trivial if and only if 
	$\mathscr G_R$ has exactly three connected components $H_1,H_2,H_3$ and 
	precisely one of the following conditions is satisfied:
	\begin{enumerate}
		\item There are  integers $p,q,r$ such that 
				$H_1\simeq K_p\Box K_q$, $H_2\simeq K_p\Box K_r$, and $H_3\simeq K_q\Box K_r$.
		\item There are  integers $p,q,r,s$ such that 
				$H_1\simeq (K_p\Box K_q) \join (K_{r}\Box K_{s})$, 
				$H_2\simeq (K_p\Box K_r) \join (K_{q}\Box K_{s})$, and 
				$H_3\simeq (K_p\Box K_s) \join (K_{q}\Box K_{r})$.
	\end{enumerate}
\end{theorem}
\begin{proof}
	Suppose that $R=\ovTheta_G$ for some connected graph $G=(V,E)$ and that $\ovTheta_G^*$ is not
	$1$-trivial. By Thm.~\ref{thm:ovTheta*-non-1trivial}, $G$ is a complete multipartite graph
	$K_{n_1,\dots, n_\ell}$ with $\ell\in \{3,4\}$. Let $V_1,\dots,V_\ell$ be the maximal independent
	sets of $G$ and put $[1:n]\coloneqq \{1,\dots,n\}$. In the following, we assume that $V_i =
	\{x_i\mid x\in [1:n_i]\}$. Hence, $\{x_i,y_j\}\in E$ if and only if $x$ and $y$ are integers such
	that $x\leq n_i$, $y\leq n_j$ and $i\neq j$. Consider now the graph representation $\mathscr G_R
	= (W,F)$ of $R$. For the sake of simplicity, we may assume w.l.o.g.\ that the vertices in $W$ are
	coordinatized, that is, we assume that each vertex $w\in W$ has coordinates $(x_i,y_j)$ where $w$
	represents the edge $\{x_i,y_j\}$ and $i < j$. Hence, the edges in $\mathscr G_R$ are of the form
	$\{(x_i,y_j),(x'_i,y'_j)\}$. 
	
	Let $i,j\in [1:\ell]$ be distinct. Observe first that each $x_i\in V_i$ is adjacent to all
	$y_j\in V_j$ in $G$ and that $\{x_i,y_j,y'_j\}$ with $y\neq y'$ induces always a shortest path in
	$G$ with edges $\{x_i,y_j\}$ and $\{x_i,y'_j\}$. By Lemma \ref{lem:SP}, $(\{x_i,y_j\},
	\{x_i,y'_j\})\in \ovTheta_G$ and thus, $\{(x_i,y_j),(x_i,y'_j)\}$ is an edge in $\mathscr G_R$
	whenever $y\neq y'$. In other words, the vertices $(x_i,y_j)$ induced a complete graph
	$K_{|V_j|}$ in $\mathscr G_R$ for each fixed $x\in [1:n_i]$ and all $y\in [1:n_j]$. Similarly,
	the vertices $(x_i,y_j)$ induced a complete graph $K_{|V_j|}$ in $\mathscr G_R$ for each fixed
	$y\in [1:n_j]$ and all $x\in [1:n_i]$. Moreover, $(x_i,y_j)$ and $(x'_i,y'_j)$ with $x\neq x'$
	and $y\neq y'$ are not adjacent in $\mathscr G_R$ since $\{x_i,y_j,x'_i,y'_j\}$ induces a $C_4$
	in $G$ with antipodal edges $e=\{x_i,y_j\}$ and $f=\{x'_i,y'_j\}$ which together with Obs.\
	\ref{obs:antipodal} implies that $(e,f)\notin \ovTheta_G$. To summarize, the subgraph $\mathscr
	G_R[W_{ij}]$ in $\mathscr G_R$ induced by $W_{ij}\coloneqq \{(x_i,y_j) \mid x\in [1:n_i], y\in
	[1:n_j]\} = V_i\times V_j$ has edges $\{(x_i,y_j),(x'_i,y'_j)\}$ if and only if $x=x'$ and $y\neq
	y'$ or $x\neq x'$ and $y = y'$. It is now a straightforward task to verify that $\mathscr
	G_R[W_{ij}]$ is connected and, in particular, $\mathscr G_R[W_{ij}]\simeq K_{|V_i|}\Box
	K_{|V_j|}$. Moreover, the sets $W_{ij}$, $1 \leq i<j \leq \ell$ form a partition of $E=W$, so
	each vertex of $\mathscr G_R$ belong to exactly one of the sets $W_{ij}$.
	
	We continue with distinguishing between the case $\ell=3$ and $\ell=4$. Assume first that
	$\ell=3$. In this case, we show that that $\mathscr G_R[W_{ij}]$ forms a connected component in
	$\mathscr G_R$ for any two distinct $i,j\in \{1,2,3\}$. Assume, for contradiction, that this is
	not the case. Hence, there is a vertex $w = (a_r,b_{s})\in W$ with $\{r,s\}\neq \{i,j\}$ (and
	thus, $(a_r,b_{s})\notin W_{ij}$) that is adjacent to at least one vertex $w' = (x_i,y_j)$ in
	$W_{ij}$ and thus, $\{w,w'\}\in F$. Since $\{r,s\}\cup \{i,j\} = \{1,2,3\}$ but $\{r,s\}\neq
	\{i,j\}$ it follows that precisely one of $r$ and $s$ is contained in $\{i,j\}$ while the other
	is not. Assume first that $r=i$ and thus, $s\neq i,j$. Since
	$d_G(y_j,b_s)=d_G(x_i,b_s)=d_G(y_j,a_i)=1$ and $d_G(a_i,x_i) \in \{0,2\}$, we have $d_G(x_i,a_i)
	+ d_G(y_j,b_s) \neq d_G(x_i,b_s) + d_G(y_j,a_i)$. Hence, $(w,w')\notin \ovTheta_G$ and thus,
	$\{w,w'\}\notin F$; a contradiction. By similar arguments, one shows that neither of the cases
	$r=j$, $s=i$ and $s=j$ can occur. Hence, $\mathscr G_R[W_{ij}]\simeq K_{|V_i|}\Box K_{|V_j|}$
	forms a connected component in $\mathscr G_R$. In summary, for $\ell=3$, $\mathscr G_R$ is the
	disjoint union of the three graphs $\mathscr G_R[W_{ij}] \simeq K_{|V_i|}\Box K_{|V_j|}$, $1 \leq
	i<j \leq 3$. Hence, $\mathscr G_R$ has exactly three connected components $H_1,H_2,H_3$ where
	$H_1 \simeq K_{|V_1|}\Box K_{|V_2|}$, $H_2 \simeq K_{|V_1|}\Box K_{|V_3|}$, and $H_3 \simeq
	K_{|V_2|}\Box K_{|V_3|}$.

	Assume now that $\ell=4$. By similar arguments as in the previous case, if there is a vertex $w =
	(a_r,b_{s})\in W$ that is adjacent to a vertex $w' = (x_i,y_j)$ in $W_{ij}$, then $\{r,s\} \cap
	\{i,j\} = \emptyset$ must hold. Hence, $\{r,s\} \cup \{i,j\}=\{1,2,3,4\}$. Put $W_{sr} = V_r\cup
	V_s$. As argued before, $\mathscr G_R[W_{ij}]\simeq K_{|V_i|}\Box K_{|V_j|}$ and $G[W_{rs}]\simeq
	K_{|V_r|}\Box K_{|V_s|}$ holds. Now let $w' = (x_i,y_j)\in W_{ij}$ and $w = (a_r,b_s)\in W_{rs}$
	be chosen arbitrarily. Since $\{r,s\} \cap \{i,j\} = \emptyset$, we have $d_G(x_i,a_r) =
	d_G(y_j,b_s) = d_G(x_i,b_s) = d_G(y_j,a_r) = 1$ and consequently, $d_G(x_i,a_r) + d_G(y_j,b_s) =
	d_G(x_i,b_s) + d_G(y_j,a_r)$. Hence, $(\{x_i,y_j\},\{a_r,b_s\})\in \ovTheta_G$ and thus,
	$\{w,w'\}\in F$. Since $\mathscr G_R[W_{ij}]$ and $\mathscr G_R[W_{rs}]$ are connected and the
	edge $\{w,w'\}$ exists it follows that $G_R[W_{ij}\cup W_{rs}]$ is connected. As argued above, no
	other vertex in $W\setminus (W_{ij}\cup W_{rs})$ can be adjacent to vertices in $W$. Hence,
	$G_R[W_{ij}\cup W_{rs}]$ forms a connected component $H$ of $G_R$. Moreover, since $w' =
	(x_i,y_j)\in W_{ij}$ and $w = (a_r,b_s)\in W_{r,s}$ have been chosen arbitrarily, it follows that
	$H \simeq G[W_{ij}]\join G[W_{rs}] \simeq (K_{|V_i|}\Box K_{|V_j|})\join(K_{|V_r|}\Box
	K_{|V_s|})$. Since the latter arguments hold for all three bipartitions $\{1,2\}\cupdot \{3,4\}$,
	$\{1,3\}\cupdot \{2,4\}$, and $\{1,4\}\cupdot \{2,3\}$ of the set $\{1,2,3,4\}$ and since
	$\cup_{1\leq i<j\leq 4} W_{ij} = W$ it follows that $\mathscr G_R$ has exactly three connected
	components $H_1,H_2,H_3$ where $H_1\simeq (K_{|V_1|}\Box K_{|V_2|})\join(K_{|V_3|}\Box
	K_{|V_4|})$, $H_2\simeq (K_{|V_1|}\Box K_{|V_3|})\join(K_{|V_2|}\Box K_{|V_4|})$ and $H_3\simeq
	(K_{|V_1|}\Box K_{|V_4|})\join(K_{|V_2|}\Box K_{|V_3|})$ which completes the proof of the
	\emph{only if}-direction.

	 For the \emph{if}-direction, suppose first that $\mathscr G_R = (W,F)$ has exactly three
	connected components $H_1,H_2,H_3$ and that there are three integers $p,q,r$ such that $H_1\simeq
	K_p\Box K_q$, $H_2\simeq K_p\Box K_r$, and $H_3\simeq K_q\Box K_r$. W.l.o.g.\ we may assume that
	$V(K_p)$, $V(K_q)$ and $V(K_r)$ are pairwise disjoint and that $V(H_1) = V(K_p)\times V(K_q)$,
	$V(H_2) = V(K_p)\times V(K_r)$ and $V(H_3) = V(K_q)\times V(K_r)$ and, therefore, $W =
	(V(K_p)\times V(K_q))\cupdot (V(K_p)\times V(K_r))\cupdot (V(K_q)\times V(K_r))$. Hence, each
	vertex in $W$ has coordinates $(x,y)$ with $x$ and $y$ not being in the same complete subgraph
	$K_h$, $h\in \{p,q,r\}$. Let $G = (V,E)$ be the graph with vertex set $V = V(K_p)\cupdot
	V(K_q)\cupdot V(K_r)$ and edges $\{x,y\}$ precisely if $(x,y)\in W$. Since $W$ does not contain
	vertices $(x,y)$ with $x,y\in K_h$, $h\in \{p,q,r\}$ it follows that there are no edges
	connecting vertices $x,y\in K_h$, $h\in \{p,q,r\}$. Hence, $V(K_p)$, $V(K_q)$ and $V(K_r)$ form
	independent sets in $G$. Now let $x\in V(K_i)$ and $y\in V(K_j)$ with $i,j\in \{p,q,r\}$ being
	distinct. By construction, either one of $V(K_i)\Box V(K_j)$ or $V(K_j)\Box V(K_i)$ is a subset
	of $W$ and therefore, $(x,y)\in W$ or $(y,x)\in W$. In either case, we have $\{x,y\}\in E$.
	Hence, $G\simeq K_{p,q,r}$ is a complete multipartite graph with $\ell=3$ independent sets. It
	remains to show that $R = \ovTheta_G$. This, however, can be done by similar arguments as in the
	\emph{only if}-direction. Moreover, since $\mathscr G_R=\mathscr G_{\ovTheta_G}$ has three
	connected components, $\ovTheta_G^*$ is not $1$-trivial. 
	
	Assume now that $\mathscr G_R = (W,F)$ has exactly three connected components $H_1,H_2,H_3$ and
	that there are four integers $p,q,r,s$ such that $H_1\simeq (K_p\Box K_q) \join (K_{r}\Box
	K_{s})$, $H_2\simeq (K_p\Box K_r) \join (K_{q}\Box K_{s})$, and $H_3\simeq (K_p\Box K_s) \join
	(K_{q}\Box K_{r})$. Similar as in the previous case, we may assume that $V(K_p)$, $V(K_q)$,
	$V(K_r)$ and $V(K_s)$ are pairwise disjoint and that $W = V(H_1)\cupdot V(H_2) \cupdot V(H_3)$
	with 
	 $V(H_1) = (V(K_p)\times V(K_q)) \cupdot (V(K_r)\times V(K_s))$, 
	 $V(H_2) = (V(K_p)\times V(K_r)) \cupdot (V(K_q)\times V(K_s))$ and 
	 $V(H_3) = (V(K_p)\times V(K_s)) \cupdot (V(K_q)\times V(K_r))$. 
	Again, let $G = (V,E)$ be the graph with vertex set $V = V(K_p)\cupdot V(K_q)\cupdot
	V(K_r)\cupdot V(K_s)$ and edges $\{x,y\}$ precisely if $(x,y)\in W$. By analogous argumentation as
	in the previous case, $V(K_h)$, $h\in \{p,q,r,s\}$ forms an independent set in $G$. Moreover, any
	two vertices $x$ and $y$ in $G$ are either contained in some $V(K_h)$ and thus, are non-adjacent
	or they are contained in distinct $K_i$ and $K_j$. In the latter case, $K_{i}\Box K_{j}$ or
	$K_{j}\Box K_{i}$ is part of one of the three connected components, i.e., the vertex $(x,y)$ or
	$(y,x)$ exists in $\mathscr G_R$. Thus, $\{x,y\}$ is an edge. Hence, $G\simeq K_{p,q,r,s}$ is a
	complete multipartite graph with $\ell=4$ independent sets. It remains to show that $R =
	\ovTheta_G$. This, however, can be done by similar arguments as in the \emph{only if}-direction.
	Moreover, since $\mathscr G_R=\mathscr G_{\ovTheta_G}$ has three connected components,
	$\ovTheta_G^*$ is not $1$-trivial.
\end{proof}

\begin{corollary}\label{cor:1-3-classes}
	$\ovTheta_G^*$ has either $1$ or $3$ equivalence classes for any graph $G$. 
\end{corollary}
\begin{proof}
	For a connected graph $G$, $\ovTheta_G^*$ is either $1$-trivial or has, by Theorem
	\ref{thm:structure-ovTheta}, exactly three equivalence classes defined by the three connected components
	of $\mathscr G_{\ovTheta_G}$.
	
	Suppose now that $G$ is disconnected. If $G$ is edge-less, then there is nothing to show. Hence,
	assume that $G$ contains edges. Assume first that only one connected component $C$ of $G$
	contains edges, i.e., all other components are $K_1$'s. Hence, $\ovTheta_{G[C]}^*$ is either
	$1$-trivial or has, by Theorem \ref{thm:structure-ovTheta}, exactly three equivalence classes.
	Since no further edges other than those in $G[C]$ exist, it follows that $\ovTheta_G^*=\ovTheta_{G[C]}^*$ is either
	$1$-trivial or has exactly three equivalence. Assume, now that that $G$ contains distinct
	connected components that contain edges. In this case, for any two edges $e$ and $f$ that are
	located in distinct connected components of $G$, we have $\Delta_{ef}=\{\infty\}$ and, therefore,
	$(e,f)\in \ovTheta_G$. It is now a straightforward task to verify that $\ovTheta_G^*$ is
	$1$-trivial. 
\end{proof}

\begin{corollary}
	To determine $\ovTheta^*_G$ for a given graph $G$, explicit knowledge about the distance function $d_G$ of $G$ 
	is not required. 
\end{corollary}
\begin{proof}
	Assume that $G$ is connected. 
	If $G$ is not a complete multipartite graph $K_{n_1,\dots, n_\ell}$ with $\ell\in \{3,4\}$, 
	then Theorem~\ref{thm:ovTheta*-non-1trivial} implies that $\ovTheta^*_G$ is $1$-trivial. 
	Suppose now that $G$ is a complete multipartite graph $K_{n_1,\dots, n_\ell}$ with $\ell\in \{3,4\}$. 
	Let $V_1,\dots,V_\ell$ be the maximal independent sets of $G$ of size $n_1,\dots, n_\ell$, 
	respectively. 
	By Theorem \ref{thm:structure-ovTheta}, $\mathscr G_R$ has exactly three connected components $H_1,H_2,H_3$. 

	Assume first that $\ell=3$. 
	As outlined in the proof of Theorem \ref{thm:structure-ovTheta}, 
	there are three integers $p=n_1,q=n_2,r=n_3$ such that 
	$H_1\simeq K_p\Box K_q$, $H_2\simeq K_p\Box K_r$, and $H_3\simeq K_q\Box K_r$. 
	In this case, 	$H_1$ represents the class of edges between vertices in $V_1$ and $V_2$, 
						$H_2$ represents the class of edges between vertices in $V_1$ and $V_3$, and
						$H_3$ represents the class of edges between vertices in $V_2$ and $V_3$.
	Hence, the equivalence classes of $\ovTheta_G^*$ are completely determined by
	the respective maximal independent sets of $G$.

	Assume now that $\ell=4$. 
	As outlined in the proof of Theorem \ref{thm:structure-ovTheta}, 
   then there are four integers $p=n_1,q=n_2,r=n_3,s=n_4$ such that 
				$H_1\simeq (K_p\Box K_q) \join (K_{r}\Box K_{s})$, 
				$H_2\simeq (K_p\Box K_r) \join (K_{q}\Box K_{s})$, and 
				$H_3\simeq (K_p\Box K_s) \join (K_{q}\Box K_{r})$.
	In this case, 	$H_1$ represents the class of edges between vertices in $V_1$ and $V_2$ and between vertices in $V_3$ and $V_4$, 
						$H_2$ represents the class of edges between vertices in $V_1$ and $V_3$ and between vertices in $V_2$ and $V_4$, and
						$H_3$ represents the class of edges between vertices in $V_1$ and $V_4$ and between vertices in $V_2$ and $V_3$. 
	Again, the equivalence classes of $\ovTheta_G^*$ are completely determined by
	the respective maximal independent sets of $G$. 
	
		To summarize, in none of the cases, 
	explicit knowledge about the distance function $d_G$ of $G$ 
	is  required to determine $\ovTheta_G^*$.

	Assume now that $G$ is not connected. Then determining the connected components 
	and checking if there are at least two connected components that contain edges can
	be done via a simple breath-first search. If $G$ has only one connected component
	$C$ that contains edges, then the equivalence classes of $\ovTheta_{G}^*= \ovTheta_{G[C]}^*$ can be classified 
	as in the latter cases. Otherwise,  if $G$ has more than one connected component that contains
	edges then analogous arguments as in the proof of Cor.\ \ref{cor:1-3-classes}
	show that $\ovTheta_{G}^*$ is $1$-trivial. Again, 
	no explicit knowledge about the distance function $d_G$ of $G$ 
	is  required to determine $\ovTheta_G^*$.
\end{proof}

\section{Summary and Outlook}

In this paper, we explored structural properties of the irreflexive complement $\ovTheta_G$
of the Djokovi\'{c}-Winkler Relation $\Theta_G$. In particular, we characterized those
$\ovTheta_G^*$ having exactly one, respectively, more than one equivalence class. Moreover,
we showed that determining $\ovTheta_G^*$ does not require explicit knowledge about
distances in the underlying graph $G$. In addition, we provided novel characterizations of block
graphs, trees, graphs with diameter less than three, and more than two.

In view of future research, it might be of interest to establish results similar to
Theorem~\ref{thm:structure-ovTheta} assuming that $R=\ovTheta_G$ and $\ovTheta_G^*$ is $1$-trivial.
Moreover, can one characterize those $\Theta$-relations for which $\Theta_G^*$ has $k>1$ equivalence
classes or for which connected components are Cartesian products similar to
Theorem~\ref{thm:structure-ovTheta}? In addition, one may ask which graphs $G$ satisfy $\ovTheta_G^*
= \Theta_G^*$. By Lemma~\ref{lem:Theta*1t}, this can  only happen if both $\ovTheta_G^*$ and $\Theta_G^*$
are $1$-trivial. 

\section*{Acknowledgments}
We thank the organizers of the 10th Slovenian Conference on Graph Theory (2023) in Kranjska Gora,
where the first ideas for this paper were drafted, while enjoying a cold and tasty red Union, or was
it a green La\v{s}ko?

\bibliographystyle{spbasic}
\bibliography{theta}

\end{document}